\documentclass[pdflatex,sn-mathphys-num]{sn-jnl}

\usepackage{graphicx}%
\usepackage{multirow}%
\usepackage{amsmath,amssymb,amsfonts}%
\usepackage{amsthm}%
\usepackage[title]{appendix}%
\usepackage{xcolor}%
\usepackage{textcomp}%
\usepackage{manyfoot}%
\usepackage{booktabs}%
\usepackage{algorithm}%
\usepackage{algorithmicx}%
\usepackage{algpseudocode}%
\usepackage{listings}%
\usepackage{upgreek}%
\usepackage{ifthen}%
\usepackage{mathscinet}%
\usepackage{enumitem}%
\usepackage{amssymb}%
\usepackage{mathrsfs}%
\usepackage{graphicx,colordvi,psfrag}%

\theoremstyle{thmstyleone}%
\newtheorem{theorem}{Theorem}
\newtheorem{proposition}[theorem]{Proposition}%

\theoremstyle{thmstyletwo}%
\newtheorem{example}{Example}%
\newtheorem{remark}{Remark}%
\newtheorem{corollary}{Corollary}
\theoremstyle{thmstylethree}%
\newtheorem{definition}{Definition}%
\newtheorem{lemma}{Lemma}
\newcommand{\R}{{\mathbb R}}
\newcommand{\N}{{\mathbb N}}
\renewcommand{\P}{\mathbb P}
\providecommand{\norm}[1]{\| #1 \|}
\renewcommand{\d}{\mathrm{d}}
\renewcommand{\div}{\mathrm{div}\,}

\newcommand{\n}{\mathrm{n}}
\renewcommand{\Re}{\mathrm{Re}}
\renewcommand{\Im}{\mathrm{Im}}
\newcommand{\dist}{\mathrm{dist}\,}

\newcommand{\B}{\mathscr{B}}
\renewcommand{\L}{\mathscr{L}}
\newcommand{\g}{\mathsf{g}}
\newcommand{\e}{\mathsf{e}}
\newcommand{\G}{\mathsf{G}}
\renewcommand{\O}{\mathsf{O}}
\newcommand{\I}{\mathbb{I}}
\renewcommand{\exp}{\mathrm{e}}
\renewcommand{\a}{\mathsf{a}}

\allowdisplaybreaks 
\begin{document}

\title[]{Asymptotic stability of solutions to semilinear evolution equations in Banach spaces}


\author[1]{\fnm{Francesco} \sur{Cellarosi}}\email{francesco.cellarosi@queensu.ca} \equalcont{These authors contributed equally to this work.}

\author[2]{\fnm{Anirban} \sur{Dutta}}\email{21ad53@queensu.ca}
\equalcont{These authors contributed equally to this work.}

\author*[3]{\fnm{Giusy} \sur{Mazzone}}\email{giusy.mazzone@queensu.ca}

\affil[1,2,3]{\orgdiv{Department of Mathematics \& Statistics}, \orgname{Queen's University}, \orgaddress{\street{48 University Ave}, \city{Kingston}, \postcode{K7L 3N6}, \state{Ontario}, \country{Canada}}}




\abstract{We prove a new linearization principle for the nonlinear stability of solutions to semilinear evolution equations of parabolic type. We assume that the set of equilibria forms a finite dimensional manifold of normally stable and normally hyperbolic equilibria. In addition, we assume that the linearized operator is the generator of an analytic semigroup (not necessarily stable). We show that if a mild solution to our evolution equation exists globally in time and remains ``close'' to the manifold of equilibria at all times, then the solution must eventually converge to an equilibrium point at an exponential rate. 

We apply our abstract results to the equations governing the motion of a fluid-filled heavy solid. Under general assumptions on the physical configuration and initial conditions, we show that weak solutions to the governing equations eventually converge to a steady state with an exponential rate. In particular, the fluid velocity relative to the solid converges to zero as $t\to\infty$ in $H^{2\alpha}_p(\Omega)$ for each $p\in [1,\infty)$ and $\alpha\in [0,1)$ as well as in $H^{2}_2(\Omega)$.}




\maketitle

\section{Introduction}
The analysis of the stability and long-time behaviour of solutions to evolution equations is a classical problem in the theory of differential equations, and finds its application to many problems arising in mechanics and in various areas of applied science. 

Consider a (nonlinear) autonomous evolution equation on a Banach space $X$ 
\begin{equation}\label{eq:intro-evolution}
\left\{\begin{aligned}
&\frac{\d u}{\d t}+f(u)=0&&\text{for }t>0,
\\
&u(0)=u_0\in X.  &&
\end{aligned}\right.
\end{equation}
Generally speaking, there are two broad methods, originally due to Lyapunov (see \cite{Lyapunov}), to study the stability of solutions to \eqref{eq:intro-evolution} ``near'' an equilibrium point (i.e., corresponding to $u_0$ ``sufficiently close'' to $\bar u$, where $\bar u$ solves $f(\bar u)=0$). The {\em direct method} is based on finding a so-called {\em Lyapunov function} $V(t)=V(u(t))$ satisfying certain monotonicity properties along solutions of \eqref{eq:intro-evolution}. For mechanical systems, $V$ often coincides with some {\em energy functional} associated to the physical system under consideration. However, for many mechanical systems (as simple as those with a finite number of degrees of freedom), finding an appropriate Lyapunov function may be a difficult task. A less direct method---which was actually the first one proposed by Lyapunov---is based on a {\em linearization principle}. In the finite-dimensional setting (e.g. if $X=\R^n$ in \eqref{eq:intro-evolution}), a linearization principle usually provides conditions under which the stability properties of equilibria of \eqref{eq:intro-evolution} can be completely determined by the spectral properties of the Jacobian matrix of the vector-valued function $f$ evaluated at the equilibrium point. This argument can be generalized in the infinite-dimensional setting; this is the main focus of our paper. 

We note that the stability properties of an equilibrium $\bar u$ of \eqref{eq:intro-evolution} are given by the corresponding stability properties of the trivial solution of 
\begin{equation}\label{eq:intro-evolution0}
\left\{\begin{aligned}
&\frac{\d u}{\d t}+Lu=G(u)&&\text{for }t>0,
\\
&u(0)=\bar u_0, &&
\end{aligned}\right.
\end{equation}
where $L:=f'(\bar u)$ is the Gateaux derivative of $f$ at $\bar u$, and $G(u):=f'(\bar u)u-f(u)$. There are numerous results that use linearization for the stability analysis of the zero solution to \eqref{eq:intro-evolution0} (respectively, to \eqref{eq:intro-evolution} if $f(0)=0$) in (infinite-dimensional) Banach spaces, see \cite{Ball73,Rostamian,DaPrato-Grisvard79,henry,Kielhofer75,Lions-heat}. These works are mainly concerned with semilinear 
evolution equations of parabolic type (see equation \eqref{eq:evolution} below). For similar results in the context of quasilinear evolution equations\footnote{These are equations of the type \eqref{eq:evolution} below, but with $A=A(u)u$.}, we refer the interested readers to \cite{Drangeid,Guidetti,Lunardi-quasilinear,Potier-FerryI,Simonett}. There is also extensive literature on stability results for evolution equations of hyperbolic type. We will not cite these works here since the current paper is focused on evolution equations of parabolic type. 

Common features of the papers cited above are that $0\in \varrho(L)$, where $\varrho(L)$ is the resolvent set of $L$, and/or that the trivial solution of $Lu=G(u)$ is ``isolated'', that is, there are no other (nonzero) solutions of $Lu=G(u)$ in a ball centered at $u=0$. The work \cite{DaPrato-Lunardi-center} uses linearization to characterize the asymptotic stability, instability or saddle point properties of the trivial solution of \eqref{eq:intro-evolution} (assuming that $f(0)=0$) in a more general setting, when the equation possesses a finite-dimensional {\em center manifold}. Specifically, in \cite[Sections 3 \& 4]{DaPrato-Lunardi-center}, it is assumed that $\lambda=0$ is a simple eigenvalue\footnote{This means that algebraic and geometric multiplicity of $\lambda=0$ is equal to 1, see \cite[Definition A.2.7]{Analytic-semigroups-Lunardi}. } of $L$, there are no eigenvalue on the imaginary axis, and the remainder of spectral points have positive real part. The books \cite{henry} and \cite{Analytic-semigroups-Lunardi} provide other results in this direction. 

More recently, \cite{PSZ09} presents a new linearization principle, that the authors call {\em generalized principle of linearized stability}, for quasilinear evolution equations characterized by a $C^1$-manifold of {\em normally stable} and/or {\em normally hyperbolic} equilibria (in the sense of the Definitions \ref{def:normally-stable} and \ref{def:normally-hyperbolic} below). In particular, $\lambda=0$ is assumed to be a {\em semi-simple} eigenvalue of $L$ (see Definition \ref{def:normally-stable}, part \ref{semi-simple}). That work provides nonlinear stability and instability results as well as conditions for the convergence, with exponential rate, of solutions to an equilibrium point (``near'' stable and unstable equilibria). While \cite{DaPrato-Lunardi-center,henry,Analytic-semigroups-Lunardi} use the semigroup approach (in particular, $L$ must be the generator of an analytic semigroup), in \cite{PSZ09} it is assumed that $L$ has the property of maximal $L^p$-regularity (which is a stronger functional setting for the existence of solutions to \eqref{eq:intro-evolution} compared to the so-called {\em mild solutions} found through the semigroup approach). 

Inspired by some fluid-solid interaction problems, the authors in \cite{RIMS,GaldiMazzone21} return to the semigroup approach and prove new linearization principles for semilinear evolution equations by replacing some of the geometric assumptions on the set of equilibria in \cite{PSZ09} with growth conditions for the nonlinear operator $G$ in \eqref{eq:intro-evolution0}. Their results are nonlinear stability and instability results, and provide conditions for the exponential convergence of solutions ``near'' a stable equilibrium. 

In this paper, we prove a more general linearization principle that applies to semilinear evolution equations with normally stable and/or normally hyperbolic equilibria, by assuming that the relevant linear operator $L$ in \eqref{eq:intro-evolution0} is the generator of an analytic semigroup. Theorem \ref{th:stable}
shows that any normally stable equilibrium is exponentially stable (in the sense of Definition \ref{def:asymptotic-stability}). Whereas, Theorem \ref{th:unstable} shows that any normally hyperbolic equilibrium is unstable. In addition, any global (in time) solution that remains ``close'' to the set of equilibria at all times must converge to an equilibrium point at an exponential rate. Our results generalize those of \cite{PSZ09} to a weaker functional setting (see Remarks \ref{rm:comaprison-PSZ09-stable},  \ref{rm:comaprison-PSZ09-unstable} and \ref{rm:X_alpha-X_p}), and they extend those of \cite{RIMS,GaldiMazzone21}, under more general assumptions on $G$ in \eqref{eq:intro-evolution0}, by proving results for exponential convergence of solution ``near'' unstable equilibra (see Remarks \ref{rm:comparison-GM21-stable} and \ref{rm:comparison-unstable}). Finally, we apply our abstract results to the equations governing the motion of a fluid-filled heavy solid. In a rather general geometrical and physical setting (see assumptions \ref{hp:physics-center},\ref{hp:physics-G}, \ref{hp:physics-lambda} and \ref{hp:physics-principal} below), we show that steady-state solutions to the equations of motion \eqref{eq:motion} are either normally stable or normally hyperbolic, and any weak solution ({\em \`a la Leray-Hopf}, in the sense of Definition \ref{def:weak-solutions}), corresponding to a large set of initial data,  converges to a steady-state solution at an exponential rate (see Theorem \ref{th:main-application}). In particular, the fluid velocity relative to the solid converges to zero as $t\to\infty$ in $H^{2\alpha}_p(\Omega)$ for each $p\in [1,\infty)$ and $\alpha\in [0,1)$ as well as in $H^{2}_2(\Omega)$. These outcomes enhance and improve the results found by the third author of this paper and collaborators in \cite{GaldiMazzone21, GaldiMazzoneMohebbi18}. 

We conclude this introduction with a brief outline of the paper. In Section \ref{sec:notation}, we fix the notation, introduce some definitions and recall some useful results. Our main results on the new linearization principles for semilinear evolution equations can be found in Section \ref{sec:main-thm}. In Section \ref{sec:application}, we present the application to the fluid-solid interaction problem of a fluid-filled heavy solid. In Appendix \ref{sec:flattening-equilibria}, we recall the reparametrization of the equilibria set introduced in \cite{PSZ09}. Finally, in Appendix \ref{app:unstable}, we reproduce a proof of the nonlinear instability property of normally hyperbolic equilibria that can be traced back to \cite{henry}.

\section{Notation and useful results}\label{sec:notation}
If $(X, \norm{\cdot}_X)$ is a Banach space, $B_{X}(x,R)$ denotes the open ball in $X$ centered at $x$ with radius $R$. For an interval $I\subset\R$ and $1\le p<\infty$, $L^p(I;X)$ (resp. $W^{k,p}(I;X)$, $k\in \N$) denotes the space of strongly measurable functions $f$ from $I$ to $X$ for which $\left(\int_I \norm{f(t)}^p_X\; d t\right)^{1/p}<\infty$ (resp. $\sum^k_{\ell=0}\left(\int_I \norm{\partial^\ell_t f (t)}^p_X\; d t\right)^{1/p}<\infty$). Similarly, $C^k(I;X)$ indicates the space of functions which are $k$-times differentiable with values in $X$, and having $\sup_{t\in I}\norm{\partial^\ell_t \cdot}_X < \infty$, for all $\ell = 0,1,...,k$. In addition, we write \( f \in C_w(I; X) \) when the map $ t \in I \mapsto \varphi(f(t)) \in \mathbb{R}$ is continuous for all bounded linear functionals \( \varphi \) defined on \( X \). 

Let $X$ and $Y$ be real Banach spaces, and $L:\;X\to Y$ be a linear operator. We will denote with $D(L)$, $N(L)$, and $R(L)$ the domain, kernel, and range of $L$, respectively. In addition, $\sigma(L)$ denotes the spectrum of $L$, and 
\[\begin{split}
&\Re(\sigma(L)):=\{\Re(\lambda):\; \lambda\in \sigma(L)\},
\\
&\Im(\sigma(L)):=\{\Im(\lambda):\; \lambda\in \sigma(L)\},
\end{split}\]
where $\Re(\lambda)$ and $\Im(\lambda)$ are the real and imaginary part of the (complex) number $\lambda$, respectively. Also, $\mathcal B(X;Y)$ denotes the space of all linear bounded operators $L:\; X\to Y$, whereas $C^1(X;Y)$ indicates the space of continuously differentiable operators from $X$ to $Y$. 

Let $A:\; D(A)\to X$, with $D(A) \hookrightarrow X$, be a linear, sectorial operator with 
$\Re (\sigma(A)) \subset (0, \infty)$. For $\alpha \in [0,1]$, the fractional powers $A^\alpha$ of $A$ are well-defined, and the following sets
\begin{align}\label{eq:X_alpha}
   X_{\alpha} := \{u \in X: \norm{u}_{X_\alpha}:=\norm{A^{\alpha} u}_X < \infty\}
\end{align}
are Banach spaces which are compactly embedded in $X$ (see \cite[Section 1.4]{henry}). In particular, $X_{0} \equiv X$, and for $u \in X$ we have that $\norm{u}_{0} \equiv \norm{u}_X$. We also recall that the following interpolation inequality holds
\begin{align}\label{eq:interpolation-X_alpha}
    \norm{u}_{X_\alpha} \leq c\norm{u}^{1-\alpha}_X\norm{Au}^{\alpha}_X.
\end{align}

Now, consider the following (nonlinear) autonomous evolution equation on a Banach space $X$:
\begin{equation}\label{eq:evolution0}
\left\{\begin{aligned}
&\frac{\d u}{\d t}+Au=F(u)&&\text{for }t>0,
\\
&u(0)=u_0. &&3
\end{aligned}\right.
\end{equation}
In the above, $A:\; X_1:=D(A)\to X$ is a linear, sectorial operator, and $F\in C^1(Y;X)$ with $X_1\hookrightarrow Y\hookrightarrow X$.  

We  recall some standard definitions.

\begin{definition}\label{def:stability-lyapunov} Given $A$ and $F$ as above, we say that $\bar{u}\in X_1$ is an {\em equilibrium} of \eqref{eq:evolution0} if $A\bar u=F(\bar{u})$. 
The equilibrium $\bar u$ of \eqref{eq:evolution0} is said to be {\em stable} (in the sense of Lyapunov) if for every $\varepsilon>0$, there exists $\delta>0$ such that for all $u_0\in Y$, with $\norm{u_0-\bar u}_Y<\delta$, the corresponding solution $u$ of \eqref{eq:evolution0} satisfies  
\[
\norm{u(t)-\bar u}_Y<\varepsilon \qquad \text{for all }t>0. 
\]
If $\bar u$ is not stable, we will say that it is {\em unstable}. 
\end{definition}
In many applications, like the one presented in Section \ref{sec:application}, the set 
\begin{equation}\label{eq:general-equilibria-set}
\mathcal E:=\{u_*\in D(A):\; Au_*=F(u_*)\}
\end{equation}
of all equilibria is not a union of isolated singletons, but it constitutes a non-trivial, finite-dimensional manifold. In view of our stability results, we modify the classical definition of {\em asymptotic stability} as follows:

\begin{definition}\label{def:asymptotic-stability}
We say that the equilibrium $\bar u$ of \eqref{eq:evolution0} is 
{\em exponentially stable in $Y$} 
if for every $\varepsilon >0$ there exists $\gamma>0$ such that for all $u_0\in Y$, with $\norm{u_0-\bar u}_Y<\gamma$, the corresponding solution $u$ of \eqref{eq:evolution0} satisfies
\begin{itemize}
    \item[(i)] $u\in C([0,\infty);Y)\cap C((0,\infty);X_1)\cap C^1((0,\infty);X)$;
    \item[(ii)] $\norm{u(t)-\bar u}_Y<\varepsilon$ for all $t>0$, i.e. $\bar u$ is stable (in the sense of Lyapunov); 
    \item[(iii)] there exist $u_*\in \mathcal E$ and $t_\infty\ge 0$ such that $\norm{u(t)-u_*}_Y<k\exp^{-\beta t}$ for all $t> t_\infty$, and some positive constants $k$ and $\beta$.
\end{itemize}
\end{definition}

Next, we recall the definitions of normally stable and normally hyperbolic equilibria from \cite{PSZ09}. 

\begin{definition}\label{def:normally-stable}
Let $\bar u\in X_1$ be an equilibrium of \eqref{eq:evolution0}, and set $L:=A-F'(\bar u)$, where $F'(\bar u)\in \mathcal B(Y;X)$ is the Gateaux derivative of $F$ at $\bar u$. The equilibrium $\bar u$ of \eqref{eq:evolution0} is said to be {\em normally stable} if the following conditions on $L$ hold. 
\begin{enumerate}[label=\rm{(\Roman*)},ref=\rm{(\Roman*)}]
    \item\label{equi1} There exist an open subset $U \subset \R^{m}$, $m\ge 1$, with $0 \in U$, and a $C^{1}$-function $\psi:\; U \to D(L)$, such that 
    \begin{enumerate}[label=\rm{(I.\alph*)},ref=\rm{(I.\alph*)}]
        \item\label{equi1a} $\psi (U) \subset \mathcal{E}$ and $\psi(0) = \bar u$;
        \item\label{equi1b} The rank of $\psi'(0)$ equals $m$;
        \item\label{equi1c} $A\psi(\zeta) = F(\psi(\zeta))$ for all $\zeta \in U$; and 
        \item\label{equi1d} $\mathcal{E} \cap B_{D(L)}(\bar u,r) = \psi(U)$ for some ball $B_{D(L)}(\bar u,r)$ of radius $r>0$.
    \end{enumerate}
    \item\label{tangent-space} The tangent space $T_{\bar u}\mathcal E$ of $\mathcal{E}$ at $\bar u$ is given by $N(L)$.
    \item\label{semi-simple} $\lambda=0$ is a {\em semi-simple eigenvalue} of L, i.e., 
    \begin{enumerate}[label=\rm{(III.\alph*)},ref=\rm{(III.\alph*)}]
        \item $0$ is an isolated eigenvalue of L; and 
        \item $N(L) \oplus R(L) = X$.
    \end{enumerate}
    \item\label{no-im} $\sigma(L) \cap i\mathbb R = \{0\}$.
    \item\label{ns} $\Re(\sigma(L))\setminus\{0\}\subset (0,+\infty)$. 
\end{enumerate}
\end{definition}

\begin{remark}\label{rm:spectral-projections-n-stable}
Assume that $\bar u$ is a normally stable equilibrium of \eqref{eq:evolution0} and $D(L)=D(A)\equiv X_1$. Then $\sigma(L)$ admits the following decomposition into two disjoint non-trivial parts  
\[
\sigma(L) = \sigma_{c}(L) \cup \sigma_{s}(L), 
\]
where $\sigma_{c}(L) :=\{0\}$ and 
\begin{equation}\label{eq:stable-spectrum}
\sigma_s(L):=\{\lambda\in \sigma(L):\; \Re(\lambda)\in (0,+\infty)\}.
\end{equation}
Following \cite[Theorems 5.7-A,B]{taylor1958}, we consider $P^{\ell}$ the spectral projection corresponding to the spectral set $\sigma_{\ell}(L)$, $\ell \in \{c,s\}$. To ease the notation, we relabel $X$ with $X_0$ (and recall that $X_1=D(L))$, then we set  
\begin{equation}\label{eq:X^ell_j}
X^{\ell}_{j} := P^{\ell}(X_{j}), \qquad\text{for }\ell \in \{c,s\},
\end{equation}
these spaces are equipped with the norms of $X_{j}$  for $j \in \{0,1\}$. 
Since $\{0\}$ is a compact spectral set, then $X^{c}_{0}\equiv X^{c}_{1} = N(L) =: X^c$ as vector spaces, and we endow $X^c$ with the norm $\norm{\cdot}_{X_0}$.\footnote{From the above definition of normally stable equilibrium, it follows that dim$(X^{c}) = m$. Since $X^{c}$ has finite dimension, all norms on it are equivalent. } We obtain the decompositions
\begin{equation}\label{eq:decomposition-n-stable}
\begin{split}
X_{0}= X^{c} \oplus X^{s}_{0},
\\
X_{1}= X^{c} \oplus X^{s}_{1},
\end{split}
\end{equation}
and with these decompositions, we can write 
\begin{equation}\label{eq:decomposition-L-n-stable}
L= L_{c} \oplus L_{S}, \qquad L_{\ell}:= L P^\ell = P^\ell L\qquad \text{for }\ell \in \{c,s\}.
\end{equation}
It follows that $\sigma(L_{\ell}) = \sigma_{\ell}(L)$ for $\ell \in \{c,s\}$. Certainly, $L_{c} \equiv 0$.
\end{remark}

In general, $L$ may have also eigenvalues with negative real part. We then denote with 
\begin{equation}\label{eq:unstable-spectrum}
\sigma_u(L):=\{\lambda\in \sigma(L):\; \Re(\lambda)\in (-\infty,0)\},
\end{equation}
the ``unstable'' part of the spectrum of $L$. We have the following definition of normally hyperbolic equilibrium.  

\begin{definition}\label{def:normally-hyperbolic}
With the same notation as in Definition \ref{def:normally-stable}, we say that the equilibrium $\bar u$ of \eqref{eq:evolution0} is {\em normally hyperbolic} if conditions (I)--(IV) of Definition \ref{def:normally-stable} hold, and condition (V) is replaced with the following one:
\begin{enumerate}[label=\rm{(V')},ref=\rm\rm{(V')}]
    \item\label{ni} $\sigma_u(L)\neq \emptyset$. 
\end{enumerate}
\end{definition}

\begin{remark}\label{rm:spectral-projections-n-hyperbolic}
Assume that $\bar u$ is normally hyperbolic, and that $L$ is sectorial with $D(L)=D(A)$. Following the same steps as in Remark \ref{rm:spectral-projections-n-stable} (thanks to \cite[Theorems 5.7-A,B]{taylor1958}), we consider the spectral projections $P^{\ell}$ corresponding to the spectral sets $\sigma_{\ell}(L)$, where $\sigma_{c}(L)= \{0\}$, whereas $\sigma_{s}(L)$ and $\sigma_{u}(L)$ are the stable and unstable parts of the spectrum of $L$ as defined in \eqref{eq:stable-spectrum} and \eqref{eq:unstable-spectrum}, respectively. Then, we set $X^{\ell}_{j} := P^{\ell}(X_{j})$ for $\ell \in \{c,u,s\}$, equipped with the norms of $X_{j}$.
As in Remark \ref{rm:spectral-projections-n-stable}, we notice that $X^{c}_{0}\equiv X^{c}_{1} = N(L) =: X^c$ as vector spaces, and we endow $X^c$ with the norm $\norm{\cdot}_{X_0}$. In addition, since the operator $-L$ generates an analytic semigroup on $X_{0}$, $\sigma_{u}(L)$ is a compact spectral set for $L$. Hence, $X^{u}_{0}\equiv X^{u}_{1}$ coincide as vector spaces. Now, $L_{u} := LP^{u}$, the part of $L$ in $X^{u}_{0}$, is invertible, so the spaces $X^{u}_{j}$  (for $j\in\{0,1\}$) carry equivalent norms. We then set $X^{u} := X^{u}_{0}= X^{u}_{1}$ and equip $X^{u}$ with the norm of $X_{0}$. We finally obtain the decompositions
\begin{equation}\label{eq:decomposition-X-n-hyperbolic}
\begin{split}
X_{0}= X^{c} \oplus X^{s}_{0} \oplus X^{u},
\\
X_{1}= X^{c} \oplus X^{s}_{1} \oplus X^{u},
\end{split}
\end{equation}
and with these decompositions, we can write 
\begin{equation}\label{eq:decomposition-L-n-hyperbolic}
L= L_{c} \oplus L_{s} \oplus L_{u},\qquad  L_{\ell} := LP^\ell = P^\ell L\qquad \text{for }\ell \in \{c,s,u\}.
\end{equation}
It follows that $\sigma(L_{\ell}) = \sigma_{\ell}(L)$ for $\ell \in \{c,s,u\}$, and again $L_{c} \equiv 0$.
\end{remark}

Under the assumptions \ref{H1} and \ref{H2} below (see also Remark \ref{rm:stability-0}), $L=A+B$ will be sectorial with $D(L)=D(A)$ and $-BA^{-\alpha}\in\mathcal B(X;X)$ for some $\alpha\in [0,1)$. Assume that $\Re(\sigma_s(L))\subset (\gamma, +\infty)$ for some $\gamma>0$. The following estimates for $L_s$ (defined in \eqref{eq:decomposition-L-n-stable} and \eqref{eq:decomposition-L-n-hyperbolic}) can be deduced from \cite[Theorem 1.5.4]{henry} and will be often used (sometime without any explicit mention) throughout the text: 
\begin{align}
\norm{\exp^{-tL_s}u}_{X_\beta}&\le c\norm{u}_X\,t^{-\beta}\exp^{-\gamma t},\label{eq:estimate-e-Ls-X}
\\
\norm{\exp^{-tL_s}u}_{X_\beta}&\le c\norm{u}_{X_\beta}\,\exp^{-\gamma t}.\label{eq:estimate-e-Ls-Xbeta}
\end{align}
These estimates hold for all $u\in X^s_0\cap D(A^\beta)$, $\beta\in [0,1]$, $t>0$, and with a positive constant $c$ independent of $u$ and $t$.\footnote{Note that we have used the same symbol for the constant $c$ as in \eqref{eq:interpolation-X_alpha}. In our estimates, we may use the same symbol for generic (positive) constants independent of the relevant variables of our problem.}

 We conclude this section, with some notation that will be used in Section \ref{sec:application}. 

We denote by $|\cdot|$ the Euclidean norm of vector fields (we will use the same symbol for the absolute value of scalar fields and for the Frobenius norm of tensor fields). For a domain $\Omega\subseteq\R^d$, $d\geq1$,  $p\in [1,\infty]$, and $k\in \N$, we denote by $L^p(\Omega)$ and $H^k_p(\Omega):=W^{k,p}(\Omega)$ the classical Lebesgue and Sobolev spaces, respectively. 
For $s\in \R$, $H^s_q(\Omega)$ denotes the Bessel potential space. We recall the characterization $H^s_q(\Omega)=[H^{s_0}_{q_0}(\Omega),H^{s_1}_{q_1}(\Omega)]_{\theta}$ as complex interpolation, which is valid for any real $s_0\ne s_1$, $q, q_0,q_1\in [1, \infty)$, $\theta\in (0,1)$, with $\frac{1}{q}= \frac{1-\theta}{q_0}+\frac{\theta}{q_1}$ and $s=(1-\theta)s_0+\theta s_1$ (see \cite[6.4.5. Theorem.]{bergh}).

We are now ready to present our new results on the nonlinear spectral stability analysis of equilibria of \eqref{eq:evolution0}.

\section{New principles of linearized stability for semilinear evolution equations}\label{sec:main-thm}

Let $(X,\norm{\cdot}_{X})$ be a Banach space, and consider the following evolution equation on $X$ 
\begin{equation}\label{eq:evolution}
\left\{\begin{aligned}
&\frac{\d u}{\d t}+Au=F(u)&&\text{for }t>0,
\\
&u(0)=u_0\in X. &&
\end{aligned}\right.
\end{equation}
We make the following assumptions on the linear operator $A$ and on the nonlinear operator $G$: 
\begin{enumerate}[label=(H\upshape\arabic*),ref=(H\upshape\arabic*)]
    \item\label{H1} $A:\;D(A)\to X$ is a linear, sectorial operator with $D(A)\hookrightarrow X$ and $\Re(\sigma(A))\subset (0, \infty)$;
    \item\label{H2}$F\in C^1(X_\alpha;X)$ for some $\alpha\in [0,1)$; we recall the definition of $X_\alpha$ in \eqref{eq:X_alpha}.
\end{enumerate}

\begin{remark}\label{rm:stability-0}
As customary in stability analysis, the study of the stability properties of an equilibrium $\bar u$ of the evolution equation \eqref{eq:evolution} can be reduced to the study of the (corresponding) stability properties of the solution $u=0$ of the following problem 
\begin{equation}\label{eq:evolution1}
\left\{\begin{aligned}
&\frac{\d u}{\d t}+Lu=G(u) &&\text{for }t>0,
\\
&u(0)=u_0-\bar u\in X, &&
\end{aligned}\right.
\end{equation}
where $Lu:=Au-F'(\bar u)u$ and $G(u):=F(u+\bar u)-F(\bar u)- F'(\bar u)u$. 

If we denote with $B:=-F'(\bar u)$, then $B\in \mathcal B(X_\alpha;X)$, that is 
    \begin{align*}
        \norm{Bu}_{X} \leq c_\alpha \norm{u}_{X_\alpha}, \qquad\text{with }\alpha \in [0,1),
    \end{align*}
where $c_\alpha$ is a positive constant. 
Under the above assumptions, we have that $D(L)=D(A)=X_1$.
From assumption \ref{H1}, \eqref{eq:interpolation-X_alpha}, and \cite[Proposition 2.4.1 (i)]{Analytic-semigroups-Lunardi} 
it also follows that $L$ is a sectorial. In addition, the estimates \eqref{eq:estimate-e-Ls-X} and \eqref{eq:estimate-e-Ls-Xbeta} hold with $\beta=\alpha$ and the notation introduced in Remark \ref{rm:spectral-projections-n-hyperbolic}. 
\end{remark}

The set of equilibria of \eqref{eq:evolution} (cf. \eqref{eq:general-equilibria-set}) is given by 
\begin{equation}\label{eq:equilibria-set}
    \mathcal{E} = \{u\in D(A):\; Au=F(u)\}. 
\end{equation} 

\begin{theorem}\label{th:stable}
Let $(A,F)$ satisfy the hypotheses \ref{H1} and \ref{H2}. If $\bar u$ is a normally stable equilibrium of \eqref{eq:evolution}, then it is exponentially stable (in the sense of Definition \ref{def:asymptotic-stability}). 
\end{theorem}

\begin{proof}
    The proof is based on the analysis of the stability of the solution $u=0$ for an evolution equation equivalent to \eqref{eq:evolution1} ---referred to as its normal form--- obtained by using the special structure of the set of equilibria $\mathcal{E}$ near a normally stable equilibrium (see part \ref{equi1} of Definition \ref{def:normally-stable}) and the spectral projections introduced in Remark \ref{rm:spectral-projections-n-stable}. Such evolution equation has been introduced in \cite[Theorem 2.1.]{PSZ09} in the context of maximal $L^p$-regularity. We reproduce its derivation here for completeness. 
Recall the decompositions \eqref{eq:decomposition-n-stable}, 
allowing us to write $L= L_{c} \oplus L_{s}$ as in \eqref{eq:decomposition-L-n-stable}.
In Appendix \ref{sec:flattening-equilibria} we construct a map $\Phi$ that will allow us to parametrize a neighbourhood of the normally stable equilibrium $\bar u$. Namely, for some $\rho_0>0$ we find $\Phi\in C^1(B_{X^{c}}(0,\rho_{0}); X_1)$
such that 
    \[
    \Phi(0)=0,\qquad \{\Phi(x)+\bar u:\; x\in B_{X^{c}}(0, \rho_{0})\}=\mathcal E\cap W,
    \]
where $W$ is a suitable neighborhood of $\bar{ u}$ in $X_1$. Define $\phi_{s}(x) := P^s\Phi(x)=\Phi(x)-x$ for $x \in B_{X^{c}}(0, \rho_{0})$. Note that
\[
\phi_{s} \in C^1(B_{X^{c}}(0, \rho_{0}); X^s_1), \quad \phi_{s}(0) = \phi_{s}'(0) = 0.
\]
The mapping $\phi_s$ describes the sub-linear part of the parametrization of the manifold $\mathcal{E}$ of equilibria near $\bar u$ via
\[
B_{X^{c}}(0, \rho_{0}) \ni x \mapsto \bar{u} + x + \phi_{s}(x)\in \mathcal E\cap W.
\]
        
    Since $\phi_s'(0)=0$ and $\phi_s$ is $C^1$, we may  assume that $\rho_0>0$ is chosen so that 
    \begin{align}\label{norm-of-phi_s'leq1}
        \lVert \phi '_{s}(x) \rVert_{\mathcal{B}(X^{c};X^s_1)} \leq 1 \mbox{ for  } x \in B_{X^{c}}(0, \rho_{0}).
        \end{align}
We observe that condition \ref{equi1c} can be equivalently rephrased as
the system of equations 
\begin{equation}
         \left\{ \begin{aligned}
        &P^{c}G( x + \phi_{s}(x))=0, &\text{}\\
        &P^{s}G( x + \phi_{s}(x))= L_{s}\phi_{s}(x), \label{split-equilibri-equ}
          \end{aligned}\right.\quad x \in B_{X^{c}}(0, \rho_{0}).
\end{equation}

Introducing the new variables 
\begin{center}
    $x := P^{c}(u- \bar u),\quad y:= P^{s}(u- \bar u)-\phi_{s}(x)$
\end{center}
and defining
\begin{align}\nonumber
    T(x,y) &:= P^{c}\left(G(x + y + \phi_{s}(x))- G(x + \phi_{s}(x))\right), &\text{}\\ \nonumber
    R_{s}(x,y) &:= P^{s}(G(x + y + \phi_{s}(x))- G(x + \phi_{s}(x)) - \phi '_{s}(x)T(x,y),
\end{align}
we rewrite \eqref{eq:evolution1} as the following system of evolution equations:
\begin{equation} \left\{\label{decomposed equation}
\begin{aligned}
\displaystyle
&\frac{\d x}{\d t} = T(x,y), \qquad &&x(0)= x_{0}, \\[5pt]
&\frac{\d y}{\d t} + L_{s}y(t) = R_{s}(x,y), \qquad && y(0)= y_{0} . \\[0pt] 
\end{aligned}\right.
\end{equation}
      Clearly, $T(x,0) = R_{s}(x,0)=0$, for $x \in B_{X^{c}}(0, \rho_{0})$, showing that the equilibria set $\mathcal{E}$ of \eqref{eq:evolution} near $\bar u$ is in one-to-one correspondence with the  set $B_{X^{c}}(0, \rho_{0}) \times \{0\} \subset X^{c} \times X^{s}_{1}$ of equilibria of \eqref{decomposed equation}.\footnote{Note that we have a one-to-one correspondence between the solutions of \eqref{eq:evolution} near $\bar u$ in $X_{\alpha}$ and solutions of \eqref{decomposed equation} near $0$.} We refer to the system \eqref{decomposed equation} as the {\em normal form} of \eqref{eq:evolution} around its normally stable equilibrium $\bar u$.

 Since $G$ is a $C^{1}$ function from $X_{\alpha}$ to $X_0$, given $\eta > 0$ (to be chosen later) we can find $r=r(\eta)\leq 5\rho_0$ small enough 
 such that
      \begin{align}\label{2.12}
          \lVert G(u_{1}) - G(u_{2}) \rVert_{X_0} \leq \eta \lVert u_{1} - u_{2} \rVert_{X_\alpha}
      \end{align}
      for all $u_{1} , u_{2} \in B_{X_{\alpha}}(0, r)$. Using  \eqref{2.12} along with the fact that $P^{c}$ and $P^{s}$ are bounded on $X_{0}$ and the bound \eqref{norm-of-phi_s'leq1},
      we obtain the estimates  
    \begin{align}\label{4.14}
        \left \lVert T(x, y) \right \rVert_{X_0} , \lVert R_{s}(x, y) \rVert_{X_0} \leq \beta\lVert y \rVert_{X_\alpha},
    \end{align}
    for all $x \in  \overline{B_{X^{c}}(0, \rho)}$ and $y \in \overline{B_{X^{s}_{\alpha}}(0, \rho)}$, 
where $\rho=\frac{r}{5}$
and 
$\beta = c_{2}\eta$ for some  absolute constant $c_{2}$. 

Let $t_{*}$ denote the maximum existence time for the solution $(x(t), y(t))$ of the system \eqref{decomposed equation}.
Consider
\begin{align}
        t_{1} = t_{1}(x_{0},y_{0}) :=  \sup \left\{t\in (0, t_{*}) \colon \lVert x(\tau) \rVert_{X_0}, \lVert y(\tau) \rVert_{X_\alpha} \leq \rho, \mbox{ for all }\tau \in [0,t]\right\}.\label{def-t_1}
    \end{align}
We claim that $t_1=t_*$. Suppose the contrary, i.e. that $t_{1} < t_{*}$. Applying the variation of parameter formula to the second equation in \eqref{decomposed equation}, for all $t \in [0, t_{1}]$ we get
\begin{equation}
    \lVert y(t) \rVert_{X_\alpha} \leq \left\lVert\exp^{-L_{s}t} y_{0}\right\rVert_{X_\alpha}+ \int_{0}^{t} \left \lVert\exp^{-L_{s}(t - s)}R_{s}(x(s),y(s)) \right\rVert_{X_\alpha}\! \d s.
\end{equation}
By the spectral properties of $L$, we know that there exists $\gamma > 0$ such that
$\Re \left(\sigma_s(L)\right) \subset (\gamma, \infty)$. 
This, along with 
\eqref{4.14} and the fact that $L$ is sectorial, yields 
\begin{align}
    \lVert y(t) \rVert_{X_\alpha} &\leq M_{\alpha} \exp^{-\gamma t} \lVert y_{0}\rVert_{X_\alpha} + \int_{0}^{t} M_{\alpha} \frac{\exp^{-\gamma (t - s)}}{(t-s)^{\alpha}} \left \lVert R_{s}(x(s), y(s))\right 
    \rVert_{X_0} \d s\nonumber\\ 
    &\leq M_{\alpha} \exp^{-\gamma t} \lVert y_{0}\rVert_{X_{\alpha}} + \int_{0}^{t} M_{\alpha} \frac{\exp^{-\gamma (t - s)}}{(t-s)^{\alpha}} \beta\,\lVert y(s) \rVert_{X_\alpha} \d s. \label{estimate 1}
\end{align}
Now we choose $b >0$ such that $\gamma - b > 0$ and $1 <p' < \frac{1}{\alpha}$. Applying  H\"{o}lder's inequality to \eqref{estimate 1}, we get, for all $t \in [0,t_{1}]$,
\begin{align*}
    \lVert y(t) \rVert_{X_{\alpha}}
    &\leq M_{\alpha} \exp^{-\gamma t} \lVert y_{0}\rVert_{X_\alpha} 
    \\
    &\qquad + \beta  M_{\alpha} \left(\int_{0}^{t} \frac{\exp^{(-\gamma p' + bp') (t - s)}}{(t-s)^{p' \alpha}}\d s\right)^{\frac{1}{p'}}
    \left(\int_{0}^{t} \exp^{- bp (t - s)} \lVert y(s) \rVert_{X_\alpha}^{p}\d s\right)^\frac{1}{p}.
\end{align*}
Since $
\displaystyle\int_{0}^{\infty}  \dfrac{ \exp^{-b s} }{ s^{\alpha}} \d s < \infty$,  we get

\begin{align*}
    \lVert y(t) \rVert_{{X_\alpha}} \leq M_{\alpha} \exp^{-\gamma t} \lVert y_{0}\rVert_{X_\alpha} + \beta M_{\alpha} \left(\int_{0}^{t} \exp^{- bp (t - s)} \lVert y(s) \rVert_{X_\alpha}^{p}\d s\right)^\frac{1}{p}.
\end{align*}
Therefore, we have
\begin{align}
    \lVert y(t) \rVert_{X_\alpha}^{p} \leq M_{\alpha}^{p} \exp^{-p\gamma t} \lVert y_{0}\rVert_{X_\alpha}^{p} + (\beta M_{\alpha})^{p} \int_{0}^{t} \exp^{- bp (t - s)} \lVert y(s) \rVert_{X_\alpha}^{p} \d s.
\end{align}
Now we choose $\eta>0$  sufficiently small in \eqref{2.12} so that  $\beta=c_2\eta$ satisfies the inequality $pb-(\beta M_{\alpha})^{p}=:\omega>0$. By Gr\"onwall's inequality we get
\begin{align}\label{2.18}
    \lVert y(t) \rVert_{X_\alpha}^{p} \leq  M_{\alpha}^{p} \lVert y_{0}\rVert_{X_\alpha}^{p} \exp^ {-\omega t},
\end{align}
for all $t \in [0,t_{1}]$.

Applying the variation of parameter formula to the first equation in \eqref{decomposed equation}, and using \eqref{4.14} together with \eqref{2.18}, we get
\begin{align}
    \lVert x(t) \rVert_{X_0} &\leq \lVert x_{0}\rVert_{X_0} +  \int_{0}^{t} \lVert T(x(s), y(s))\rVert_{X_0} \d s\nonumber\\
    &\leq\lVert x_{0}\rVert_{X_0} +  \int_{0}^{t}\beta\lVert y(s) \rVert_{X_\alpha}  \d s \nonumber \\
    &\leq \lVert x_{0}\rVert_{X_0} +  \beta M_{\alpha} \lVert y_{0}\rVert_{X_\alpha} \int_{0}^{t}  \exp^{-\frac{\omega t}{p}} \d s \nonumber\\
    &\leq \lVert x_{0}\rVert_{X_0} + \beta  M_{\alpha} \lVert y_{0}\rVert_{X_\alpha}.\label{eq:estimate-stable-x}\end{align}
Now we choose the initial data $x_0, y_0$ in \eqref{decomposed equation} sufficiently small to ensure that
\begin{align}
    M_{\alpha}\lVert y_{0}\rVert_{X_\alpha} +  2\left(\lVert x_{0}\rVert_{X_0} + \dfrac{p \beta M_{\alpha} \lVert y_{0}\rVert_{X_\alpha} }{\omega}\right) \leq \frac{\rho}{2}.\label{eq:stable-assumption-small-initial-conditions}
\end{align}

Since $t_{1} < t_{*}$, by the continuity of the solution $u(t)$ in $X_\alpha$ and the estimates \eqref{2.18}--\eqref{eq:estimate-stable-x}--\eqref{eq:stable-assumption-small-initial-conditions}, we get
\begin{align}
\lVert x(t_{1}) \rVert_{X_0}, \lVert y(t_{1}) \rVert_{X_{\alpha}}  \leq \frac{\rho}{2} < \rho.
\end{align}
This contradicts the definition of $t_{1}$. Therefore, we have $ t_{*}= t_{1}$.

{We use \eqref{def-t_1} to argue that $t_*= \infty$. In fact, if the opposite was true, i.e. $t_{*} < \infty$, then $\lVert u(t) \rVert_{X_\alpha} \rightarrow \infty$ as $t \rightarrow t_{*}^-$}.

We have shown that the solution $u(t)$ of \eqref{eq:evolution1} exists globally and for all $t \geq 0$ we have
\begin{align}
&\lVert x(t) \rVert_{X_0}, \lVert y(t) \rVert_{X_\alpha} \leq \rho,\label{x_and_y_in_rho_ball}\\
&\lVert y(t) \rVert_{X_\alpha}^{p} \leq  M_{\alpha}^{p} \lVert y_{0}\rVert_{X_\alpha}^{p} \exp^{-\omega t}.\label{y(t)_decays_exponentially_fast}
\end{align}

To show that $u(t)$ is stable, note that \eqref{x_and_y_in_rho_ball} gives
\begin{equation}\label{u(t)_u_bar_leq_C_alpha_rho}
    \begin{aligned}
    \lVert u(t)-\overline{u}\rVert_{X_\alpha}&\leq \lVert x(t)\lVert_{X_0}+\lVert y(t)\rVert_{X_\alpha}+\lVert \phi_s(x(t))\rVert_{X_\alpha}\\
    &\leq \lVert x(t)\rVert_{X_0} +\lVert y(t)\rVert_{X_\alpha}+C_\alpha \lVert \phi_s(x(t))\rVert_{X_1}\\
    &\leq C_\alpha\lVert x(t)\rVert_{X_0}+\lVert y(t)\rVert_{X_\alpha}\\
    &\leq C_\alpha\rho
    \end{aligned}
\end{equation}
for every $t\geq 0$, provided \eqref{eq:stable-assumption-small-initial-conditions} holds.

By \eqref{y(t)_decays_exponentially_fast},  
we note that  $\displaystyle\int_{0}^{\infty} \lVert T(x(s), y(s))\rVert_{X_0} \d s$ is a convergent integral due to the exponential decay of $\lVert y(t) \rVert_{X_\alpha}$. We define
\begin{equation}
    x_{\infty} := \lim\limits_{t \to \infty} x(t) = x_{0} + \int_{0}^{\infty} T(x(s), y(s)) \d s
\end{equation}
and we obtain the following estimates:
\begin{align}
        \lVert x(t) - x_{\infty} \rVert_{X_0} &\leq \int_{t}^{\infty} \lVert T(x(s), y(s))\rVert_{X_0} \d s \\
        &\leq \beta \int_{t}^{\infty} \lVert y(s) \rVert_{X_\alpha}  \d s \\\label{2.22}
       &\leq  \lVert y_{0}\rVert_{X_\alpha} \beta \exp^{-\frac{\omega t}{p}}.
\end{align}
In addition, by \eqref{norm-of-phi_s'leq1}, we have \begin{equation}\label{2.23}
    \begin{aligned}
        \lVert \phi_{s}(x(t)) - \phi_{s}(x_{\infty}) \rVert_{X_\alpha} \leq  C_\alpha \lVert x(t) - x_{\infty} \rVert_{X_0} \leq  C_\alpha \lVert y_{0}\rVert_{X_\alpha} \beta \exp^{-\frac{\omega t}{p}}.
    \end{aligned}
\end{equation}

Set $u_\infty :=\bar u + x_{\infty} + \phi_{s}(x_{\infty}) \in \mathcal{E}$. 
Using \eqref{2.22} and \eqref{2.23}, we conclude the proof by observing that  
\begin{equation}\label{u(t)-u_inf_conclusion}
    \begin{aligned}
        \lVert u(t) - u_{\infty} \rVert_{X_\alpha} &\leq \lVert x(t) - x_{\infty} \rVert_{X_0} + \lVert \phi_{s}(x(t)) - \phi_{s}(x_{\infty}) \rVert_{X_\alpha} + \lVert y(t) \rVert_{X_\alpha} \\
        &\leq C_\alpha \lVert y_{0}\rVert_{X_\alpha} \exp^{-\frac{\omega t}{p}}.
    \end{aligned}
\end{equation}
\end{proof}

\begin{remark}\label{rm:comaprison-PSZ09-stable}
Our result can be rephrased as a ``generalized principle of linearized stability'', comparable to the one proved in \cite[Theorem 2.1.]{PSZ09}. In our case, after linearizing our evolution equation near an equilibrium (see Remark \ref{rm:stability-0}), the main linear operator $L$ is sectorial. Even though the result \cite[Theorem 2.1.]{PSZ09} deals with more general (quasi-linear) evolution equations, the authors assume that the corresponding linearization $L$ possesses the property of maximal \( L^p \)-regularity (which is a stronger assumption than ours). Furthermore, after linearization, our nonlinear term satisfies 
$G \in C^1(X_\alpha; X_0)$, whereas the functional space characterizing the domain of their nonlinear term requires higher regularity on the solution to \eqref{eq:evolution1} (in particular, the variation of parameter formula can not be applied). 
\end{remark}

\begin{remark}\label{rm:comparison-GM21-stable}
    A similar exponential stability result for the solution $u=0$ of \eqref{eq:evolution1} has been proved in \cite[Theorem 2.2]{GaldiMazzone21}. In that paper, the authors assume that the nonlinear operator $G$ can be controlled by the $X_\alpha$-norm of the stable component of the solution in a sufficiently small neighborhood of the origin, see \cite[condition (H5)]{GaldiMazzone21}. We do not need such a condition here because of the structure of our set of equilibria (in particular, because of Definition \ref{def:normally-stable} parts \ref{equi1}-\ref{tangent-space}). 
\end{remark}

We are now ready to discuss our instability result. 

\begin{theorem}\label{th:unstable}
Let $(A,F)$ satisfy the hypotheses \ref{H1} and \ref{H2}. If $\bar u$ is a normally hyperbolic equilibrium of \eqref{eq:evolution}, then it is unstable.

Additionally, for every $\rho > 0$ there exists $\delta\in (0,\rho]$ such that the unique solution $u$ of \eqref{eq:evolution} corresponding to $u_0 \in B_{X_\alpha}(\bar u,\delta)$, in the regularity class  
\[
C([0,t^*);X_\alpha)\cap C((0,t^*);D(A))\cap C^1((0,t^*);X)
\]
for some maximal existence time $t^*\in (0,+\infty]$, satisfies at least one of the following two conditions:
\begin{enumerate}[label=(\alph*),ref=(\alph*)]
    \item\label{ua1}  there exists a time $t_0\in (0,t^*)$ such that $\dist_{X_{\alpha}}(u(t_{0}), \mathcal{E}) > \rho$; 
    \item\label{ua2} $t^*= +\infty$ and there exists $u_\infty\in \mathcal E $ such that 
    \begin{equation}\label{eq:exp-convergence-unstable}
    \norm{u(t)-u_\infty}_{X_\alpha}< c\, \exp^{-kt}\qquad \text{for all }{t>t_\infty}
    \end{equation}
    {for some $t_\infty\ge 0$}, and positive constants $c$ and $k$.
\end{enumerate}
\end{theorem}
\begin{proof}
    The proof of the instability property follows from \cite[Theorem 5.1.5]{henry}, under more general assumptions on $A$ and $F$. We reproduce its proof in Appendix \ref{app:unstable}, see Theorem \ref{th:unstable-correct-henry}, because some versions of \cite{henry} seem to have a typo in their proof. 

    It remains to prove the characterization of the long-time behaviour of solutions starting near a normally hyperbolic equilibrium, as stated in the second part of the theorem. We recall the decompositions \eqref{eq:decomposition-X-n-hyperbolic} for the spaces $X_0= X$ and $X_1= D(A)=D(L)$, respectively, and the corresponding decomposition \eqref{eq:decomposition-L-n-hyperbolic} for the operator $L$. For the rest of this proof, the spaces $X_j$ are endowed with the following norms    
    \begin{equation}\label{equivalent-norm}
        \norm{v}_{X_j}= \norm{P^cv}_{X_0} + \norm{P^sv}_{X_j} + \norm{P^uv}_{X_0}, \quad j=0, \alpha, 1.
    \end{equation}
    In addition, for $\ell \in \{c, u\}$, the spaces $X^\ell_j$ are equipped with the norm of $X^\ell_0$ (see Remark \eqref{rm:spectral-projections-n-hyperbolic}). 

    Similarly to the proof of Theorem \ref{th:stable}, we need a suitable reparametrization of the set of equilibria in a neighborhood of $\bar u$ (see Appendix \ref{sec:flattening-equilibria}). Specifically, let $\Phi\in C^1(B_{X^{c}}(0, \rho_{0}); X_1)$ be the mapping obtained in Lemma \ref{lem:re-parametrization}. Recall that 
    \[
    \Phi(0)=0,\qquad \{\Phi(x)+\bar u:\; x\in B_{X^{c}}(0, \rho_{0})\}=\mathcal E\cap W,
    \]
    where $W$ is a suitable neighborhood of $\bar{ u}$ in $X_1$. We define 
    \[
    \phi_{\ell} := P^\ell\Phi\in C^1(B_{X^{c}}(0, \rho_{0}); X^\ell_1)\qquad\text{ for $\ell \in \{s, u\}$}.  
    \]
    Since $\phi_{\ell}(0) = \phi_{\ell}'(0) = 0$, we take $\rho_0$ small enough so that 
    \begin{equation}\label{eq:norm-of-phi_ell'leq1}
    \sup_{w\in B_{X^{c}}(0, \rho_{0})}\lVert \phi '_{\ell}(w)\rVert_{\mathcal{B}(X^c; X^{1}_{\ell})} 
    <\min\left\{1,\frac{1}{C_\alpha}\right\} 
    \quad \mbox{for  } 
    \ell \in \{s, u\},
    \end{equation}
    where $C_\alpha>0$ is a constant for the embedding $X_{1} \hookrightarrow X_{\alpha}$, i.e. $\norm{v}_{X_\alpha}\le C_\alpha\norm{v}_{X_1}$, for all $v\in X_1$. We also rewrite \ref{equi1c} as the following system of equations
    \begin{equation}\label{split-equilibri-equ-unstable}
          \begin{aligned}
        &P^{c}G( x + \phi_{s}(x)+ \phi_u(x))=0, &\text{}\\
        &P^{\ell}G( x + \phi_{s}(x)+ \phi_u(x))= L_{\ell}\phi_{\ell}(x), \quad x \in B_{X^{c}}(0, \rho_{0}) \mbox{ and } \ell \in \{s, u\}.
          \end{aligned}
    \end{equation}
    Introducing the new variables 
    \begin{center}
    $x := P^{c}(u- \bar u),\quad y:= P^{s}(u- \bar u)-\phi_{s}(x)$, $\quad z:= P^{u}(u- \bar u)-\phi_{u}(x)$, 
    \end{center}
    and defining  
    \begin{align}\nonumber
   T(x,y,z) &:= P^{c}(G(x + y + z +\phi_{u}(x) + \phi_{s}(x))- G(x +\phi_{u}(x) + \phi_{s}(x))), &\text{}\\
    R_{\ell}(x,y,z) &:= P^{\ell}(G(x + y + z +\phi_{u}(x) + \phi_{s}(x))- G(x +\phi_{u}(x) + \phi_{s}(x)))\nonumber
    \\
    &\qquad\qquad\qquad\qquad\qquad\qquad\qquad\qquad - \phi '_{\ell}(x)T(x,y,z), \qquad \ell \in \{s,u\}, 
    \end{align}
    we obtain the following {\em normal form}\footnote{Since $T(x,0,0) = R_{\ell}(x,0,0)=0$, for all $x \in B_{X^{c}}(0, \rho_{0})$ and $ \ell \in \{s, u\}$, the equilibria set $\mathcal{E}$ of \eqref{eq:evolution0} near $\bar u$ is in one-to-one correspondence with the set $B_{X^{c}}(0, \rho_{0}) \times \{0\} \times \{0\} \subset X^{c} \times X^{s}_{1} \times X^{u}$ of equilibria of \eqref{decomposed equation-unstable}.
    } of \eqref{eq:evolution} around its normally hyperbolic equilibrium $\bar u$: 
    \begin{equation} \label{decomposed equation-unstable}
    \left\{\begin{aligned}
    \displaystyle
    &\frac{\d x}{\d t} = T(x,y,z), \qquad &&x(0)= x_{0}, \\
    &\frac{\d y}{\d t} + L_{s}y(t) = R_{s}(x,y,z), \qquad && y(0)= y_{0}, \\
    &\frac{\d z}{\d t} + L_{u}y(t) = R_{u}(x,y,z), \qquad && z(0)= z_{0}. \\
    \end{aligned}\right. 
    \end{equation}
    Since $G$ is a $C^{1}$ function from $X_{\alpha}$ to $X_0$, given $\eta > 0$ (to be chosen later) we can find $r=r(\eta)\leq {(3+2C_\alpha)}\rho_0$ 
    small enough such that
    \begin{align}\label{3.27}
          \lVert G(u_{1}) - G(u_{2}) \rVert_{X_0} \leq \eta \lVert u_{1} - u_{2} \rVert_{X_\alpha}
      \end{align}
      for all $u_{1} , u_{2} \in B_{X_{\alpha}}(0, r)$. Using  \eqref{3.27} along with the fact that $P^{c}$, $P^{s}$ and $P^{u}$ are bounded on $X_{0}$ and the bound \eqref{eq:norm-of-phi_ell'leq1}, we obtain the estimates  
    \begin{align}\label{2.58}
        \lVert T(x, y, z) \rVert_{X_0} , \lVert R_{\ell}(x, y, z) \rVert_{X_0} \leq \beta(\lVert y \rVert_{X_\alpha} + \lVert z \rVert_{X_0}),
    \end{align}
    for all $x \in  \overline{B_{X^{c}}(0, \hat\rho)}$ and $y \in \overline{B_{X^{s}_{\alpha}}(0, \hat \rho)}$, where {$\hat \rho=\frac{r}{3+2C_\alpha}$} 
    and $\beta = c_{2}\eta$ for some  absolute constant $c_{2}$.

    Let $\rho\in(0, \frac{\rho_0}{4}]$, and consider the unique solution $u$ of \eqref{eq:evolution} corresponding to $u_0 \in B_{X_{\alpha}}(\bar u,\delta)$ for some $\delta\in (0,\rho]$ to be determined later. 
    Let $t_{*}$ denote the maximum existence time of $u(t)$ 
    (and correspondingly of $(x(t), y(t), z(t))$ for the system \eqref{decomposed equation-unstable}). We have the following {\em blow-up criterion}: 
    \begin{itemize}
        \item[(i)] either $t_*=\infty$,  
        \item[(ii)] or $\lVert u(t) \rVert_{X_\alpha} \rightarrow \infty$ as $t \rightarrow t_{*}^-$.
    \end{itemize}
    Assume that part (b) of the theorem holds, that is, $t_{*} = \infty$ and 
    \[
    \dist_{X_{\alpha}}(u(t), \mathcal{E}) \le \rho\qquad\text{for all }t\in (0,\infty). 
    \]
    We first show that $u(t)\in \overline{B_{X_\alpha}(\bar u,3\rho)}$ for all $t\in [0,\infty)$ provided that $\delta=\rho$. Consider
    \begin{align*}
        t_{1} := t_{1}(x_{0},y_{0},z_{0}) :=  \sup \{t\in (0, \infty) \colon \lVert u(\tau) - \bar u \rVert_{X_\alpha} \leq 3\rho, \mbox{ for all }\tau \in [0,t]\},
    \end{align*}
    we then show that $t_1=\infty$. 
    
    For all $t \in [0, t_{1}]$, we write 
    \begin{equation}\label{eq:decomposition-c-s-u}
    u(t) = \bar u + x(t) + y(t) + z(t) + \phi_{s}(x(t)) + \phi_{u}(x(t)),
    \end{equation}
    and by \eqref{equivalent-norm}, we immediately obtain that 
    \begin{equation}\label{eq:estimate-t_1-x-unstable}
    \lVert x(t) \rVert_{X_0} \leq 3\rho.
    \end{equation}
    Using \eqref{eq:norm-of-phi_ell'leq1}, we get 
    \begin{align}
        \norm{y(t)}_{X_\alpha}&\le \norm{y(t)}_{X_\alpha}+\norm{x(t)}_{X_0}-C_\alpha\norm{\phi_{s}(x(t))}_{X_1}\nonumber
        \\
        &
        \le \norm{y(t)-\phi_{s}(x(t))}_{X_\alpha}+\norm{x(t)}_{X_0}\nonumber
        \\
        &\le \norm{u(t) - \bar u}_{X_\alpha}\le 3\rho. \label{eq:estimate-t_1-y-unstable}
    \end{align}  
    
    With an analogous reasoning, we also obtain that 
    \begin{equation}\label{eq:estimate-t_1-z-unstable}
    \lVert z(t) \rVert_{X_0} \leq 3\rho.
    \end{equation}

    Our next step is to improve the estimates \eqref{eq:estimate-t_1-y-unstable} and \eqref{eq:estimate-t_1-z-unstable}. Since $\mathcal{E}$ is a finite-dimensional manifold, for each $v \in {B_{X_{\alpha}}(\bar u, 3\rho)}$ there is $v_* \in \mathcal{E}$ such that $\dist_{X_{\alpha}}(v, \mathcal{E}) = \lVert v - v_* \rVert_{X_\alpha}$ and,  by triangle inequality, $v_*\in B_{X_{\alpha}}(\bar u, 4\rho)$. Thus, for each $t\in [0,t_1]$, there exists $u_*=u_*(t)\in \mathcal E\cap B_{X_{\alpha}}(\bar u, 4\rho)$ such that 
    \[
    \dist_{X_{\alpha}}(u(t), \mathcal{E}) = \lVert u(t) - u_*(t) \rVert_{X_\alpha},
    \]
    and $u_*(t)=\bar u+x_*(t)+\phi_s(x_*(t))+\phi_u(x_*(t))$. From \eqref{eq:decomposition-c-s-u}, it then follows that 
    \begin{align*}
    \norm{z(t)}_{X_0}&\leq \lVert x(t)-x_*(t) \rVert_{X_0} + \lVert z(t) \rVert_{X_0} -\lVert \phi_{u}(x(t)) -\phi_{u}(x_*(t)) \rVert_{X_0}
    \\
    &\le \lVert x(t) - x_*(t) \rVert_{X_0} + \lVert z(t) + \phi_{u}(x(t)) -\phi_{u}(x_*(t)) \rVert_{X_0}
    \\
    &\le \lVert u(t) - u_*(t) \rVert_{X_\alpha}=\dist_{X_{\alpha}}(u(t), \mathcal{E})\le \rho. 
    \end{align*}
    Therefore, we get the ``improved'' estimate 
    \begin{equation}\label{eq:improved-z-unstable}
    \lVert z(t) \rVert_{X_0} \leq \rho\quad\text{for all $t \in [0, t_{1}]$} 
    \end{equation}
    and, using the same argument, one shows also the ``improved'' estimate 
    \begin{equation}\label{eq:improved-y-unstable}
        \lVert y(t) \rVert_{X_\alpha} \leq \rho\quad\text{ for all $t \in [0, t_{1}]$}.   
    \end{equation}
    We prove that $t_1=\infty$ by showing that $\norm{u(t_1)-\bar u}_{X_\alpha}<3\rho$. Note that,  thanks to \eqref{eq:decomposition-c-s-u}, \eqref{equivalent-norm}, \eqref{eq:improved-z-unstable} and \eqref{eq:improved-y-unstable}, it is enough to show that 
    \begin{equation}\label{bound-on-x(t)}
        \lVert x(t) \rVert_{X_0} < c\rho \quad\mbox{ for all $t \in [0, t_{1}]$}
        \end{equation}
        for some constant $c\in (0,1)$.
   Using the variation of parameter formula to the second equation in \eqref{decomposed equation-unstable}, together with the properties of $L$, we find that 
    \begin{equation}\label{2.60}
    \lVert y(t) \rVert_{X_{\alpha}} \leq M_{\alpha} \exp^{-\gamma t} \lVert y_{0}\rVert_{X_{\alpha}} + \int_{0}^{t} M_{\alpha} \frac{\exp^{-\gamma (t - s)}}{(t-s)^{\alpha}} \lVert R_{s}(x(s), y(s), z(s))\rVert_{X_{0}}\; \d s, 
    \end{equation}
    where $\gamma > 0$ is such that $\sigma_s(L) = \sigma(L_s)  \subset (\gamma, \infty)$. Thanks to the estimates \eqref{eq:improved-z-unstable} and \eqref{eq:improved-y-unstable}, we can estimate the right-hand side of the latter displayed equation using \eqref{2.58}. Integrating the resulting estimate over the interval $(0,t_1)$, we obtain
    \begin{align}
    \int_{0}^{t_{1}} \lVert y(t) \rVert_{X_{\alpha}} \d t \leq \frac{M_{\alpha}}{\gamma} \lVert y_{0}\rVert_{X_{\alpha}} + M_{\alpha}\beta  \int_{0}^{t_{1}}\int_{0}^{t} \frac{\exp^{-\gamma (t - s)}}{(t-s)^{\alpha}}\left(\lVert y(s) \rVert_{X_\alpha} + \lVert z(s) \rVert_{X_0}\right)\; \d s \d t.
    \end{align}
    By Fubini's theorem and since $\displaystyle \int_{0}^{\infty} \dfrac{\exp^{-\gamma s}}{s^{\alpha}}\d s < \infty$, we have
    \begin{align}
        \int_{0}^{t_{1}} \lVert y(t) \rVert_{X_{\alpha}} \d t &\leq \frac{M_{\alpha}}{\gamma} \lVert y_{0}\rVert_{X_{\alpha}} + M_{\alpha}\beta  \int_{0}^{t_{1}}\int_{s}^{t_{1}} \frac{\exp^{-\gamma (t - s)}}{(t-s)^{\alpha}}(\lVert y(s) \rVert_{X_\alpha} + \lVert z(s) \rVert_{X_0})\; \d t \d s \nonumber
        \\
        &\leq \frac{M_{\alpha}}{\gamma} \lVert y_{0}\rVert_{X_{\alpha}} + M_{\alpha}\beta  \int_{0}^{t_{1}}\int_{s}^{\infty} \frac{\exp^{-\gamma (t - s)}}{(t-s)^{\alpha}}(\lVert y(s) \rVert_{X_\alpha} + \lVert z(s) \rVert_{X_0})\; \d t \d s \nonumber
        \\
         &\leq \frac{M_{\alpha}}{\gamma} \lVert y_{0}\rVert_{X_{\alpha}} + M_{\alpha}\beta  \int_{0}^{t_{1}}(\lVert y(s) \rVert_{X_\alpha} + \lVert z(s) \rVert_{X_0}) \; \d s.\label{2.65}
    \end{align}
    
    Let $\omega \in (0,\mathrm{inf} \sigma_u(-L)) \cap (0,1)$ be such that 
    \begin{equation}\label{eq:exp-decay-L_u}
    \lVert \exp^{L_{u}t} \rVert_{\mathcal B(X_0;X_0)} \leq c \exp^{-\omega t}\qquad\text{for all $t\in \R$. }
    \end{equation}
    Applying the variation of parameter formula to the third equation in \eqref{decomposed equation-unstable}, for all $t \in [0, t_{1}]$ we get
    \begin{align}
        \lVert z(t) \rVert_{X_0} &\leq \lVert\exp^{L_{u}(t_1-t)} z(t_{1})\rVert_{X_0} + \int_{t}^{t_{1}} \lVert\exp^{L_{u}(s - t)}R_{s}(x(s), y(s), z(s))\rVert_{X_0}\;
        \d s\nonumber 
        \\
        &\leq M_{\alpha} \exp^{-\omega (t_1-t)} \rho + M_{\alpha}\beta\int_{t}^{t_{1}} \exp^{-\omega(s - t)} (\lVert y(s) \rVert_{X_\alpha} + \lVert z(s) \rVert_{X_0}) \; \d s.\label{2.67}
    \end{align}
    Integrating the above equation over $(0,t_{1})$ and using Fubini's theorem,  we get
    \begin{align}
    \int_{0}^{t_{1}} \lVert z(t) \rVert_{X_0} \d t &\leq \frac{M_{\alpha}}{\omega} \rho   + M_{\alpha}\beta \int_{0}^{t_{1}} \int_{t}^{t_{1}} \exp^{\omega(t - s)} (\lVert y(s) \rVert_{X_\alpha} + \lVert z(s) \rVert_{X_0}) \; \d s \d t\nonumber 
    \\
    &\leq \frac{M_{\alpha}}{\omega} \rho   + M_{\alpha}\beta \int_{0}^{t_{1}} \int_{0}^{s} \exp^{\omega(t - s)} (\lVert y(s) \rVert_{X_\alpha} + \lVert z(s) \rVert_{X_0}) \; \d t \d s\nonumber 
    \\
    &\leq \frac{M_{\alpha}}{\omega} \rho   + \frac{M_{\alpha}\beta}{\omega} \int_{0}^{t_{1}} (\lVert y(s) \rVert_{X_\alpha} + \lVert z(s) \rVert_{X_0})\; \d s.\label{2.69}
    \end{align}
   Adding \eqref{2.65} and \eqref{2.69} side by side, and choosing $\eta$ in \eqref{3.27} small enough so that $\frac{M_{\alpha}\beta}{\omega}<1$, we obtain the following estimate: 
    \begin{align}\label{eq:estimate-int-t1-y-z}
    \int_{0}^{t_{1}} (\lVert y(s) \rVert_{X_\alpha} + \lVert z(s) \rVert_{X_0})\d s \leq \frac{2 M_{\alpha}}{\gamma} \lVert y_{0}\rVert_{X_\alpha} + \frac{2 M_{\alpha}}{\omega} \rho. 
    \end{align}
    
    Applying the variation of parameter formula to the first equation in \eqref{decomposed equation-unstable}, using \eqref{2.58} and \eqref{eq:estimate-int-t1-y-z}, we get 
    \begin{equation}
    \begin{aligned}
        \lVert x(t) \rVert_{X_0} &\leq \lVert x_{0}\rVert_{X_0} +  \int_{0}^{t} \lVert T(x(s), y(s), z(s))\rVert_{X_0} \;\d s
        \\
        &\leq \lVert x_{0}\rVert_{X_0} +  \beta \left(\frac{2 M_{\alpha}}{\gamma} \lVert y_{0}\rVert_{X_\alpha} + \frac{2 M_{\alpha}}{\omega} \rho\right)
    \end{aligned}
    \end{equation}
    for all $t \in [0, t_{1}]$. If we choose the initial data $(x_0,y_0,z_0)$ and $\eta$ in \eqref{3.27} sufficiently small so that
    \begin{equation}
    \begin{aligned}
         \lVert x_{0}\rVert_{X_0} +  \beta \left(\frac{2 M_{\alpha}}{\gamma} \lVert y_{0} \rVert_{X_\alpha} + \frac{2 M_{\alpha}}{\omega} \rho\right) \leq \frac{\rho}{6}. 
    \end{aligned}
    \end{equation}
    We have therefore shown \eqref{bound-on-x(t)} with $c=\frac{1}{6}$.
    We finally obtain that 
    \[
    \lVert u(t_{1}) -\bar u \rVert_{X_\alpha} \leq 3\lVert x(t_{1}) \rVert_{X_0} + \lVert y(t_{1}) \rVert_{X_\alpha} + \lVert z(t_{1}) \rVert_{X_0} \leq \frac{3\rho}{6} + 2\rho < 3\rho
    \]
    which contradicts the definition of $t_{1}$. Therefore, the unique solution $u$ of \eqref{eq:evolution} corresponding to $u_0 \in B_{X_{\alpha}}(\bar u,\delta)$, with $\delta \leq \rho$, exists globally and satisfies 
    \begin{equation}\label{eq:uniform-u-3rho-unstable}
    u(t)\in \overline{B_{X_{\alpha}}(\bar u, 3\rho)}\quad\text{ for all }t\in [0,\infty).
    \end{equation}

    To conclude our proof, we need to prove the exponential convergence \eqref{eq:exp-convergence-unstable}. For all $t > 0$, we apply the variation of parameter formula once more together with the properties of $L$ and \eqref{2.58} to get
    \begin{equation}\label{2.62}
    \lVert y(t) \rVert_{{X_\alpha}} \leq M_{\alpha} \exp^{-\gamma t} \lVert y_{0}\rVert_{X_\alpha} + \beta M_{\alpha} \int_{0}^{t} \frac{\exp^{-\gamma (t - s)}}{(t-s)^{\alpha}}(\lVert y(s) \rVert_{X_\alpha} + \lVert z(s) \rVert_{X_0})\; \d s.
    \end{equation}
    Since $\lVert z(t) \rVert_{X_0} \leq 3\rho$ for all $t\ge 0$, by the variation of parameter formula applied to $z(t)$ along with \eqref{eq:exp-decay-L_u} and and \eqref{2.58}, we obtain 
       \begin{align}
        \lVert z(t) \rVert_{X_0} \leq  \beta M_{\alpha} \int_{t}^{\infty} \exp^{\omega(t - s)}(\lVert y(s) \rVert_{X_\alpha} + \lVert z(s) \rVert_{X_0}) \; \d s.
    \end{align}
    
    Choose $b >0$ and $p' > 1$ such that $\gamma - b > 0$, $\omega - b > 0$, and $\alpha p' < 1$. Applying H\"{o}lder's inequality to the right-hand side of \eqref{2.62}, and using the fact that $ \displaystyle\int_{0}^{\infty} \dfrac{\exp^{-b s}}{s^{\alpha}}\d s < \infty$ for all $\alpha\in (0,1)$ and $b>0$, we get
    \begin{align*}
    \lVert y(t) \rVert_{{X_\alpha}}
    &\leq M_{\alpha} \exp^{-\gamma t} \lVert y_{0}\rVert_{X_\alpha} \\
    + \beta  M_{\alpha} &\left(\int_{0}^{t} \frac{\exp^{-(\gamma - b)p' (t - s)}}{(t-s)^{p' \alpha}}\; \d s\right)^{\frac{1}{p'}}
    \left(\int_{0}^{t} \exp^{- bp (t - s)} (\lVert y(s) \rVert_{X_\alpha} + \lVert z(s) \rVert_{X_0})^{p}\; \d s\right)^\frac{1}{p}
    \\
    &\leq M_{\alpha} \exp^{-\gamma t} \lVert y_{0}\rVert_{X_\alpha} + \beta M_{\alpha} \left(\int_{0}^{t} \exp^{- bp (t - s)} \left(\lVert y(s) \rVert_{X_\alpha} + \lVert z(s) \rVert_{X_0}\right)^{p}\; \d s\right)^\frac{1}{p}.
    \end{align*} 
    Therefore, we obtain 
    \begin{align}\label{2.77}
    \lVert y(t) \rVert_{X_\alpha}^{p} \leq M_{\alpha}^{p} \exp^{-p\gamma t} \lVert y_{0}\rVert_{X_\alpha}^{p} + (\beta M_{\alpha})^{p} \int_{0}^{t} \exp^{- bp (t - s)} (\lVert y(s) \rVert_{X_\alpha}^{p} + \lVert z(s) \rVert^{p}_{X_0})\; \d s, 
    \end{align}
    and in a similar fashion, 
    \begin{align}\label{2.78}
    \lVert z(t) \rVert^{p}_{X_0} \leq (\beta M_{\alpha})^{p} \int_{t}^{\infty} \exp^{ bp (t - s)} (\lVert y(s) \rVert_{X_\alpha}^{p} + \lVert z(s) \rVert^{p}_{X_0})\; \d s.
    \end{align}
    Adding \eqref{2.77} and \eqref{2.78} side by side, and using the fact that $\gamma > b$, we find   
    \begin{align}\label{eq:estimate-y-z^p}
     \lVert y(t) \rVert_{X_\alpha}^{p} + \lVert z(t) \rVert^{p}_{X_0} \leq M_{\alpha}^{p}\left( \exp^{-pb t} \lVert y_{0}\rVert_{X_\alpha}^{p} + \beta^{p} \int_{0}^{\infty} \exp^{- bp \lvert t-s \rvert} (\lVert y(s) \rVert_{X_\alpha}^{p} + \lVert z(s) \rVert^{p}_{X_0}\; \d s\right)
     \end{align}
     for all $t \geq 0$. If we set  
     \begin{equation*}
     \begin{aligned}
        f(t):=& \lVert y(t) \rVert_{X_\alpha}^{p} + \lVert z(t) \rVert^{p}_{X_0},
        \\
        g(t) :=&  M_{\alpha}^{p} \exp^{-pb t} \lVert y_{0}\rVert_{X_\alpha}^{p} + (\beta M_{\alpha})^{p} \int_{0}^{\infty} \exp^{- bp \lvert t-s \rvert} (\lVert y(s) \rVert_{X_\alpha}^{p} + \lVert z(s) \rVert^{p}_{X_0})\;\d s,
    \end{aligned}
    \end{equation*}
    the estimate \eqref{eq:estimate-y-z^p} reads simply as $f(t)\le g(t)$ for all $t \geq 0$. 
    Thanks to \eqref{eq:decomposition-c-s-u} and \eqref{eq:uniform-u-3rho-unstable}, we claim that $g(t)$ is the bounded positive solution of the linear non-homogeneous second-order differential equation
    \begin{align}\label{2.81}
        -\frac{\d ^{2}g}{\d t^{2}} + (pb)^2 g(t) = 2pb(\beta M_{\alpha})^{p}f(t).
    \end{align}
    We subtract the term $2pb(\beta M_{\alpha})^{p}g(t)$ from both sides of the above equation and 
    we get 
    \begin{align}\label{2.82}
    -\frac{\d ^{2}g}{\d t^{2}} + b_{1}^2 g(t) = 2b_{1}(\beta M_{\alpha})^{p}f_{1}(t),
    \end{align}
    where we have taken $\eta$ in \eqref{3.27} so that $\beta=c_2\eta$ satisfies $(pb)^{2} - 2pb(\beta M_{\alpha})^{p} > 0$, and we have denoted 
    \[\begin{split}
        &b_{1} := ((pb)^{2} - 2pb(\beta M_{\alpha})^{p})^{\frac{1}{2}}
        \\
        &f_{1}(t) := \frac{pb}{b_1}(f(t)-g(t)). 
    \end{split}\]
    By the superposition principle, every bounded solution of \eqref{2.82} has the following representation for all $t \geq 0$
    \begin{align}
        g(t) = C \exp^{-b_{1} t} + (\beta M_{\alpha})^{p} \int_{0}^{\infty} \exp^{- b_1 \lvert t-s \rvert} f_{1}(s)\; \d s
    \end{align}
    where $C>0$ is an absolute constant. Since $f_{1} \leq 0$, we conclude that 
    \begin{align}\label{eq:exp-decay-y-z-unstable}
    \lVert y(t) \rVert_{X_\alpha}^{p} + \lVert z(t) \rVert^{p}_{X_0} = f(t) \leq g(t) \leq C \exp^{-pb_{1} t},
    \end{align}
    for all $t \geq 0$, and this shows that the norms $\lVert y(t) \rVert_{X_\alpha}$, $\lVert z(t) \rVert_{X_0} \rightarrow 0$ exponentially fast as $t \rightarrow \infty$. 
    
    The exponential decay just proved along with \eqref{2.58} imply that the following limit exists 
    \begin{equation*}
    x_{\infty} := \lim\limits_{t \to \infty} x(t) = x_{0} + \int_{0}^{\infty} T(x(s), y(s), z(s)) \d s,\end{equation*}
    and the following decay holds for all $t\ge 0$
    \begin{equation*}
    \begin{aligned}
        \lVert x(t) - x_{\infty} \rVert_{X_0} &\leq \int_{t}^{\infty} \lVert T(x(s), y(s), z(s))\rVert_{X_0} \;\d s 
        \\
        &\leq \beta \int_{t}^{\infty} (\lVert y(s) \rVert_{X_\alpha} + \lVert z(s) \rVert_{X_0}) \d s\\ &\leq C\beta \exp^{-\frac{bt}{p}}. 
    \end{aligned}
    \end{equation*}
    Since $\phi_\ell\in C^1(B_{X^{c}}(0, \rho_{0}); X^\ell_1)$ for $\ell \in \{s, u\}$, and \eqref{eq:norm-of-phi_ell'leq1} holds, we also have that 
    \[
    \begin{split}
        &\lVert \phi_{s}(x(t)) - \phi_{s}(x_{\infty}) \rVert_{X_\alpha} \leq \lVert x(t) - x_{\infty} \rVert_{X_0} \leq C\beta \exp^{-\frac{bt}{p}}
        \\
        &\lVert \phi_{u}(x(t)) - \phi_{u}(x_{\infty}) \rVert_{X_0} \leq \lVert x(t) - x_{\infty} \rVert_{X_0} \leq C\beta \exp^{-\frac{bt}{p}}.
    \end{split}
    \]
    Set $u_\infty:=\bar u+x_\infty+\phi_s(x_\infty)+\phi_u(x_\infty)$. By the decomposition \eqref{eq:decomposition-c-s-u}, the definition of the norm $\norm{\cdot}_{X_\alpha}$ from \eqref{equivalent-norm}, and all the above exponential decays, we conclude that 
    \begin{equation*}
    \begin{aligned}
    \norm{u(t) -\bar u}_{X_\alpha}&
    \le \norm{x(t)-x_\infty}_{X_0}+\norm{y(t)}_{X_\alpha}+\norm{z(t)}_{X_0}
    \\
    &\quad +\norm{\phi_s(x)-\phi_s(x_\infty)}_{X_\alpha}
    +\norm{\phi_u(x)-\phi_u(x_\infty)}_{X_0}
    \\
    &\leq c\, \exp^{-kt},
    \end{aligned}
    \end{equation*}
    for some positive constants $c$ and $k$, and this shows \eqref{eq:exp-convergence-unstable}.
\end{proof}

\begin{remark}\label{rm:comaprison-PSZ09-unstable}
The proof of the above theorem relies on the following two fundamental ideas:
\begin{enumerate}
    \item Obtaining a uniform bound of the solution $u$ to \eqref{eq:evolution0} 
    which depends only on the size of the initial data (see \eqref{eq:uniform-u-3rho-unstable}).
    \item Explicit exponential decays for the ``stable'' and ``unstable components'' of $u$, namely for $y(t)$ and $z(t)$, respectively (see \eqref{eq:exp-decay-y-z-unstable}).
\end{enumerate}
For part (1), the proof follows similar steps as in the proof of \cite[Theorem 6.1]{PSZ09}. The proof of (2) is completely different than its counterpart in \cite[Theorem 6.1]{PSZ09},  and it is a suitable generalization to the infinite dimensional setting of a similar results proved in the context of ordinary differential equations in $\R^n$ \cite[Satz 10.5.1.]{Pruss_Mathias_book}.    
\end{remark}

Note that the conditions (a) and (b) in Theorem \ref{th:unstable} are not mutually exclusive, as shown in the following finite dimensional example.
\begin{example}
Consider the following system of nonlinear 1st-order ODEs
     \begin{equation}\label{finite-dim-example}
\left\{\begin{aligned}
&\frac{\d x}{\d t}= 0,
\\
&\frac{\d y}{\d t}=  y^2 - y.
\\
\end{aligned}\right.
\end{equation}
The equilibria set of \eqref{finite-dim-example} is $\{(x,0): x\in \R\} \cup \{(x,1): x\in \R\}$ and 
$L := \begin{pmatrix}
0 & 0 \\
0 & 1
\end{pmatrix}$ is the linearization around the equilibrium $(0,1)$ according to Remark \ref{rm:stability-0}. It is easy to verify that both conditions \eqref{ua1} and \eqref{ua2} of Theorem \ref{th:unstable} are satisfied. Nevertheless, 
if the initial condition is \( (x(0), y(0)) = (0, c) \), where \( 0 < c < 1 \), then the corresponding solution \( (x(t), y(t)) \) converges to \( (0,0) \) exponentially fast as \( t \to \infty \).
\end{example}
\begin{remark}\label{rm:comparison-unstable}
    For the proof of the first part of Theorem \ref{th:unstable} (concerning the instability property for normally hyperbolic equilibria), in \cite[Theorem 2.3.]{GaldiMazzone21}, the authors assume a growth condition on the nonlinear operator $G$, i.e. $\norm{G(u)}\le c\norm{u}_{X_\alpha}^p$ for some $p>1$ (see \cite[conditions (H4)$_2$]{GaldiMazzone21}). As our proof shows, such a condition is not required. 
\end{remark}
Combining the previous two theorems, we obtain the following result.
\begin{corollary}\label{long time behaviour}
     Let $0 < \alpha < 1$, and consider $(A,F)$ in \eqref{eq:evolution} so that hypotheses \ref{H1} and \ref{H2} are satisfied. Let $u_*$ be a normally stable or normally hyperbolic equilibrium of \eqref{eq:evolution}. Assume that,  
     for each sufficiently small $\rho > 0$, there exists $0 < \delta \leq \rho$ such that the unique solution $u(t)$  of \eqref{eq:evolution0}, with $u(0) \in B_{X_{\alpha}}(u_{*}, \delta)$, exists for all $t \in \mathbb R^{+}$ and $\dist_{X_{\alpha}}(u(t), \mathcal{E}) \leq \rho$ for all $t\in \mathbb R^{+}$. Then there exists $u_\infty\in \mathcal E $ such that $\|u(t)-u_\infty\|_{X_\alpha}<c \exp^{-kt}$ for some positive constants $c$ and $k$.
\end{corollary}

We conclude this section with a remark about our choice of the topology for the convergences in our theorems. 
\begin{remark}\label{rm:X_alpha-X_p}
The limiting conditions in Theorem \ref{th:stable}, Theorem \ref{th:unstable} and Corollary \ref{long time behaviour} hold in the topology $X_\alpha$ of the domains of the fractional powers of $A$ (recall equations \eqref{eq:X_alpha} and \eqref{eq:interpolation-X_alpha}). This is the ``natural'' topology when dealing with {\em mild solutions} of \eqref{eq:evolution}, i.e., for solutions satisfying the variation of parameter formula 
\[
u(t)=\exp^{-tA}u_0+\int^t_0\exp^{-(t-s)A}F(u(s))\;\d s.
\]
If $A$ admits bounded imaginary powers, then we can use \emph{complex} interpolation to write
$X_\alpha=[X,X_1]_\alpha$, see \cite[Theorem 4.17.]{MR3753604}. In the applications for which the latter characterization does not hold, it may be difficult to characterize $X_\alpha$ in terms of classical functional spaces like Sobolev or Bessel potential spaces. In such situation, we recall the following embeddings in term of \emph{real} interpolation spaces (see \cite[Proposition 4.7., Proposition 1.4.]{MR3753604}):
\[
(X,X_1)_{1-\frac{1}{q},q}\hookrightarrow X_\alpha\hookrightarrow (X,X_1)_{1-\frac{1}{p},p}\quad\text{for }\:\alpha\in (0,1) \text{ and }\: 1-\frac 1p <\alpha<1-\frac 1q. 
\]
In practical applications (say when $X$ and $X_1$ are Lebesgue and Sobolev spaces, respectively), those real interpolation spaces can be characterized in terms of Besov spaces, and then our convergences can be re-stated in those topologies. 
\end{remark}
\section{An application: long-time behaviour of the motions of a fluid-filled heavy rigid body}\label{sec:application}

In this section, we will apply the stability results obtained in Section \ref{sec:main-thm} to characterize the long-time dynamics of a system constituted by a heavy rigid body with an interior cavity completely filled by a viscous incompressible fluid. 

We consider the physical system $\mathscr{S}$ constituted by a rigid body $\B$ with an interior cavity completely filled by a viscous incompressible Newtonian fluid $\L$. Mathematically, 
$\B$ occupies the bounded domain $\Omega_B\subset\R^3$ enclosed between two closed smooth surfaces\footnote{Indeed, it is enough that $\Gamma_2$ is of class $C^3$. } $\Gamma_1$ and $\Gamma_2$ with $\Gamma_1\cap\Gamma_2=\emptyset$. The fluid occupies the bounded domain $\Omega\subset\R^3$ with boundary $\partial \Omega:=\Gamma_2$. 

\begin{figure}
\begin{center}
    \psfrag{1}{$\e_1$}
    \psfrag{2}{$\e_2$}
    \psfrag{3}{$\e_3$}
    \psfrag{g}{$\g$}
    \psfrag{B}{$\Omega_B$}
    \psfrag{L}{$\Omega$}
    \psfrag{O}{$\O$}
    \psfrag{G}{$\G$}
    \includegraphics[width=.7\textwidth]{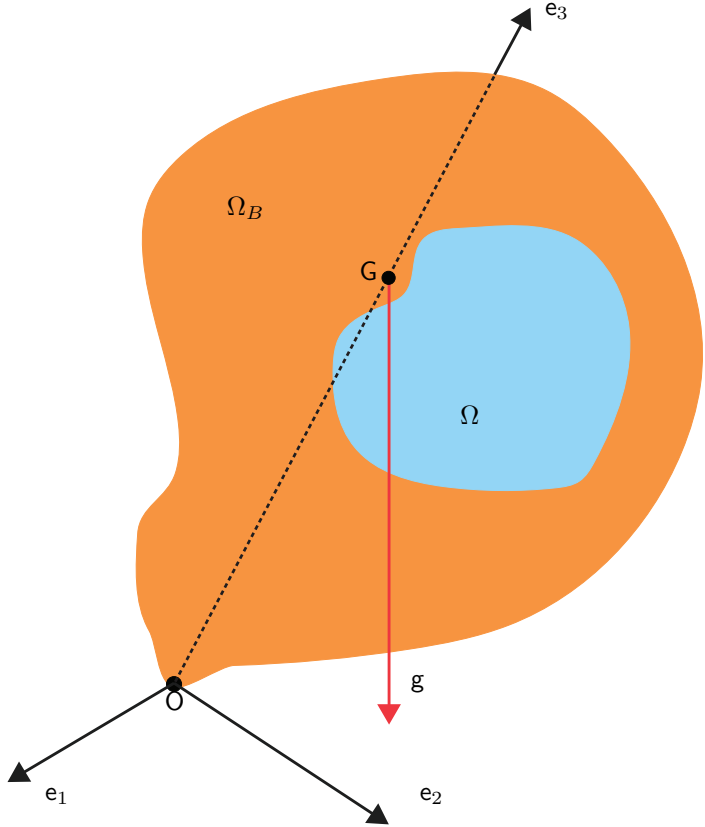}
    \caption{ A fluid-filled heavy solid}\label{fig:top}
\end{center}
\end{figure}
The motion of $\mathscr{S}$ occurs under the action of the gravity $\g$, while keeping one of its points, denoted by $\O$, fixed at all times (with respect to an inertial frame, see Figure \ref{fig:top}). As customary for problems involving moving domains\footnote{As the motions of $\B$ and $\L$ are unknowns of the problem, the domains occupied by the fluid and the solid are time-independent, and their time-evolution is also unknown. }, we choose to write the equations of motion with respect to a moving frame $\{\O;\e_1,\e_2,\e_3\}$ with origin at the fixed point $\O$ and axes directed along the {\em principal axes of inertia} of $\mathscr{S}$. To be more precise, let $\I$ denote the inertia tensor of $\mathscr{S}$ with respect to $\O$; this is a symmetric, positive definite tensor. Then, $\{\e_1,\e_2,\e_3\}$ forms an orthonormal basis of $\R^3$ of eigenvectors of $\I$ with corresponding eigenvalues $\lambda_1$, $\lambda_2$ and $\lambda_3$. We may refer to $\{\O;\e_1,\e_2,\e_3\}$ as {\em principal frame of inertia}. Displacements with respect to $\O$ will be denoted by $x$. In order to simplify the mathematical model of the physical problem, we make the following assumptions:
\begin{enumerate}[label=(\roman*)]
\item\label{hp:physics-center} The center of mass $\G$ of $\mathscr{S}$ belong to the axis passing through $\O$ along $\e_3$.
\item\label{hp:physics-G} $\G\neq \O$.
\item\label{hp:physics-principal} The principal frame of inertia $\{\O;\e_1,\e_2,\e_3\}$ is oriented so that the only non-zero coordinate $\ell$ of $\G$ is positive. 
\item\label{hp:physics-lambda} Geometry and mass distribution with $\mathscr{S}$ is such that $\{\lambda_1,\lambda_2,\lambda_3\}$ do not satisfy the following condition $\lambda_1=\lambda_2\ne \lambda_3$. 
\end{enumerate}
Therefore, with respect to the principal frame of inertia $\{\O;\e_1,\e_2,\e_3\}$, the equations governing the motion of  
$\mathscr{S}$ are given by (see \cite{Ma2})
\begin{equation}\label{eq:motion}
\left\{
\begin{aligned}
   &\div v=0\qquad&&\text{in }\Omega\times(0,T),
   \\
   &\partial_t v+\dot\omega\times x+v_i\partial_{x_i}v+2\omega\times v=\nu\Delta v-\nabla p+g\gamma \qquad&&\text{in }\Omega\times(0,T),
   \\
   &\I\cdot(\dot\omega-\dot \a)+\omega\times\I\cdot(\omega-\a)=\beta^2\e_3\times \gamma \qquad &&\text{in }(0,T),
   \\
   &\dot\gamma+\omega\times\gamma=0\qquad &&\text{in }(0,T),
   \\
   &v=0\qquad&&\text{on }\partial \Omega\times (0,T). 
\end{aligned}
\right.    
\end{equation}
In the above equations, $v$, $\rho$, and $\nu=\mu/\rho$ denote the fluid velocity relative to the solid, (constant) density, and coefficient of kinematic viscosity (with $\mu>0$ shear viscosity coefficient), respectively. The vector field $\omega$ denotes the angular velocity of $\B$, $\g:=g\gamma$ is the acceleration of gravity, with 
\begin{equation}\label{eq:gamma}
    |\gamma|=1\qquad\text{for all }t\ge0. 
\end{equation}
We note that the direction of the gravity $\gamma$ is a time-dependent (unknown) vector because the equations of motions are written with respect a non-inertial, moving frame. In addition, we set 
\begin{equation}\label{eq:beta&a}
   \beta^2:=Mg\ell,\qquad \a:=-\I^{-1}\cdot\int_\Omega \rho\,x\times v\; dx, \qquad p:=\frac1\rho\mathsf{p}-\frac 12 |\omega\times x|^2,
\end{equation}
where $M$ is the total mass of $\mathscr{S}$, and $\mathsf{p}$ is the fluid pressure, i.e., the Lagrangian multiplier due to the fluid incompressibility condition. In \eqref{eq:motion} and throughout the text, we use Einstein summation convention over repeated indices.

Equations \eqref{eq:motion}$_{1,2}$ are the Navier-Stokes equations, with additional terms given by Coriolis and centrifugal forces. These additional terms arise because the equations are written with respect to a moving frame and the chosen fluid variable is the fluid velocity relative to the solid. Equation \eqref{eq:motion}$_3$ is the balance of angular momentum of the whole system $\mathscr{S}$ with respect to the fixed point $\O$. Equation \eqref{eq:motion}$_4$ is a Poisson-type equation for the direction of the gravity $\gamma$; this ensures that gravity remains a constant vector in an inertial frame of reference. The boundary condition \eqref{eq:motion}$_5$ is the so-called {\em no-slip} boundary condition. This homogeneous Dirichlet condition physically means that the cavity of the solid is impermeable, and fluid particles adhere to the cavity during the motion. 

Our next objective is to rewrite \eqref{eq:motion} as the evolution equation \eqref{eq:evolution} on a Banach space $X$. Let $p\in (1,\infty)$ and set 
\[
L^p_\sigma(\Omega):=\{v\in L^p(\Omega):\; \div v=0\quad\text{in }\Omega,\ v\cdot \n=0\quad\text{on }\partial \Omega\}, 
\]
where the divergence condition holds in the sense of distributions, while the boundary condition holds in the weak sense. Here, and throughout the text, $\n$ denotes the unit, outward normal to $\Omega$. We then define
\begin{equation}\label{eq:def-X}
X:=\{u:=(v,\omega,\gamma)\in L^p_\sigma(\Omega)\times\R^3\times\R^3\},
\end{equation}
with norm 
\[
\norm{u}_0:=\left(\norm{\rho^{1/p}v}_{L^p(\Omega)}^2+\omega\cdot\I\cdot \omega+|\gamma|^2\right)^{1/2}. 
\]
When $p=2$, $X$ is a Hilbert space with the inner product
\[
\langle u_1,u_2\rangle:=\int_\Omega\rho\, v_1\cdot v_2\; dx+\omega_1\cdot\I\cdot\omega_2+\gamma_1\cdot\gamma_2,
\]
and associated norm $\norm{u}_0$ defined above (with $p=2$). We also introduce the Banach spaces 
\begin{align}
\label{eq:def-X1}
X_1&:=\left(H^2_p(\Omega)\cap {_0}H^1_p(\Omega)\cap L^p_\sigma(\Omega)\right)\times\R^3\times\R^3,
\\
\label{eq:def-Hsq0sigma}
_{0}H^s_{p,\sigma}(\Omega)&:=\left\{\begin{aligned}
&\{u\in H^s_p(\Omega)\cap L^q_\sigma(\Omega):\; u=0\text{ on }\partial \Omega\}\quad&&\text{ if }s>1/p,
\\
&H^s_p(\Omega)\cap L^p_\sigma(\Omega)\quad &&\text{ if }s\in [0,1/p),
\end{aligned}\right.
\end{align}
and norms
\[\begin{aligned}
    \norm{u=(v,\omega,\gamma)}_{X_1}&:=\norm{v}_{H^2_p(\Omega)}+|\omega|+|\gamma|,
    \\
    \norm{v}_{_{0}H^s_{p,\sigma}(\Omega)}&:=\norm{v}_{H^s_{p}(\Omega)},
\end{aligned}
\]
respectively. Consider the operators: 
\begin{equation}\label{eq:operators}
\begin{split}
&K_1:\; v\in L^p_\sigma(\Omega)\mapsto K_1v:=-\I\cdot\a\in \R^3
\\
&K_2:\; \omega\in \R^3\mapsto K_2\omega:=\P(\omega\times x)\in L^p_\sigma(\Omega)
\\
&E:\;
u=(v,\omega,\gamma)\in X\mapsto\left(v+K_2\omega,\I\cdot\omega+K_1v,\gamma\right)\in X,
\\
&\tilde A:\; u=(v,\omega,\gamma)\in D(A):=X_1\mapsto \tilde Au:=(-\nu\P\Delta v,\omega,\gamma)\in X,
\\
&\tilde F:\; u=(v,\omega,\gamma)\in D(F):= {_{0}H}^{2\alpha}_{p,\sigma}(\Omega)\times \R^3\times\R^3\subset X
\\
&
\mapsto  \tilde F(u):=\left(\P\left(-v_i\partial_{x_i}v-2\omega\times v\right),\omega -\omega\times\I\cdot(\omega-\a)+\beta^2\e_3\times\gamma,\gamma-\omega\times \gamma\right)\in X,
\end{split}
\end{equation}
where $\P$ denotes the Helmholtz projection of $L^p(\Omega)$ onto $L^p_\sigma(\Omega)$, whereas $I$ identifies the identity operator. We note that both $K_1$ and $K_2$ are bounded operators with finite rank. In particular, $K_2$ is compact on $L^p_\sigma(\Omega)$. So, $E$ is a compact perturbation of the operator
\[
u\in X\mapsto (v,\I \cdot \omega, \gamma)\in X. 
\]
Our claim is that $E$ is invertible. By Fredholm theory, it is enough to show that $E$ is injective. Let $u=(v,\omega,\gamma)\in X$ be such that $Eu=0$. Then, $v=-\P(\omega\times x)$ and it is also a smooth function. In particular, $v\in L^2_\sigma(\Omega)$, and 
\[
\begin{split}
0&=\langle Eu,u\rangle=\rho\norm{v}_{L^2(\Omega)}^2+\int_\Omega \rho v\cdot(\omega\times x)\; dx+
\I\cdot(\omega-\a)\cdot\I\cdot\omega+\gamma\cdot\gamma
\\
&=\rho\norm{v}_{L^2(\Omega)}^2+2\norm{\P(\omega\times x)}_{L^2(\Omega)}^2+
|\I\cdot\omega|^2+|\gamma|^2.
\end{split}\]
The latter implies that $v=\omega=\gamma=0$. Also, it has been shown in \cite[Section 6.2.3]{KoKr} that $E$ is positive definite when $p=2$. 

The nonlinear operator $\tilde F$ is bilinear and bounded in the given domain for 
\begin{equation}\label{eq:range-alpha}
\max\left\{\frac 12,\frac{3}{2p}\right\}<\alpha<1. 
\end{equation}

With the above observations, we can finally rewrite \eqref{eq:motion} as the following evolution equation on the Banach space $X$ (defined in \eqref{eq:def-X}): 
\begin{equation}\label{eq:motion-abstract}
\left\{\begin{aligned}
&\frac{\d u}{\d t}+A u = F(u)\quad &&t>0,
\\
&u(0)=u_0\quad &&
\end{aligned}\right.\qquad 
\end{equation}
where $u=(v,\omega,\gamma)^T$, $A:=E^{-1}\tilde A$ and $F:=E^{-1}\tilde F$ with 
\[
D(A)=D(\tilde A)=X_1\hookrightarrow X,\qquad D(F)=D(\tilde F)\subset X.
\]
The following proposition holds. 

\begin{proposition}[Hypothesis \ref{H1}]\label{prop:A-satisfies-H1}
    Let $p\in (1,\infty)$. Then, $A$ is sectorial on $X$ (defined in \eqref{eq:def-X}) and $\Re(\sigma(A)) \subset (0,\infty)$.
\end{proposition}
\begin{proof}
By \cite[Theorem 9.]{MR3916775}, $-\tilde A$ is the generator of a bounded analytic semigroup on $X$ and $\sigma(\tilde A)\subset(0,\infty)$. We also note that \( E^{-1} = I + K \), where \( K := -(E - I)E^{-1} \). Since \( E - I \) is a bounded operator with finite rank, it follows that \( K \) is also bounded with finite rank. By \cite[Proposition 2.4.3]{Analytic-semigroups-Lunardi}, we conclude that \( A=E^{-1} \tilde{A} \) is sectorial.

It remains to verify that \( \Re(\sigma(A)) \subset (0, \infty) \). It is well known that \( \tilde{A} \) has a compact resolvent. Consequently, the operator \( A^{-1} = \tilde{A}^{-1} E \) is compact, and thus \( A \) has a discrete spectrum consisting entirely of eigenvalues. Furthermore, we notice that the spectrum is independent of the parameter \( p \) and $\tilde A$ is positive definite (when $p=2$).
\end{proof}
We also note that $X_\alpha= D(F)$ since $A$ admits bounded imaginary
powers by \cite[Theorem 4.17.]{MR3753604}, \cite[Theorem 2]{MR786549}, and \cite[Proposition B.3. \& Remark B.4. (R2)]{gfree}. Thus, hypothesis \ref{H2} is also satisfied. 

A complete characterization of steady-state solutions\footnote{We use the terminology of {\em steady-state solution} instead of {\em equilibrium point} because we are considering time-independent solutions of a system of equations --\eqref{eq:motion}-- governing a mechanical system that may not necessarily be in mechanical equilibrium. } to \eqref{eq:motion} has been proved in \cite[Theorem 4.7]{GaldiMazzoneMohebbi18}. We recall the result in the following. 

\begin{theorem}\label{th:steady-states}\cite[Theorem 4.7]{GaldiMazzoneMohebbi18}
Let $\mathcal E:=\{u\in X:\; Au=F(u)\}$. Then, we have the following characterization of $\mathcal E$ depending on the geometry and mass distribution within $\mathscr{S}$. 
\begin{enumerate}
\item If $\lambda_1=\lambda_2=\lambda_3$, then 
\begin{equation}\label{eq:def-PR}
\mathcal E\equiv \mathsf{PR}:=\{u=(0,\omega,q)\in D(A): \;q\times \e_3=0,\; \omega=\alpha q\text{ for some }\alpha\in \R\}.
\end{equation}
\item If  $\lambda_1=\lambda_2\neq \lambda_3$, then $\mathcal E\equiv \mathsf{PR}\cup \mathsf{SP}$, where 
\begin{equation}\label{eq:def-SP}
\mathsf{SP}:=\left\{u=(0,\omega,q)\in D(A):\;\omega=\alpha q,\; q_3=-\frac{\beta^2}{\alpha^2(\lambda_3-\lambda_1)},\; \text{ for some }\alpha\in \R
\right\}. 
\end{equation}
\item If $\lambda_1\neq \lambda_2=\lambda_3$, then $\mathcal E\equiv \mathsf{PR}\cup \mathsf{SP}_1$, where 
\begin{equation}\label{eq:def-SP1}
\mathsf{SP}_1:=\left\{u=(0,\omega,q)\in D(A):\;\omega=\alpha q,\; \text{some }\alpha \in \R, \;q_2\equiv 0,\;\ q_3=-\frac{\beta^2}{\alpha^2(\lambda_3-\lambda_1)}
\right\}. 
\end{equation}
\item If $\lambda_1=\lambda_3\neq \lambda_2$, then $\mathcal E\equiv \mathsf{PR}\cup \mathsf{SP}_2$, where 
\begin{equation}\label{eq:def-SP2}
\mathsf{SP}_2:=\left\{u=(0,\omega,q)\in D(A):\;\omega=\alpha q,\; \text{some }\alpha\in \R, \;q_1\equiv 0,\;\ q_3=-\frac{\beta^2}{\alpha^2(\lambda_3-\lambda_2)}
\right\}. 
\end{equation}
\item If $\lambda_1\neq \lambda_2\neq \lambda_3$, then $\mathcal E\equiv \mathsf{PR}\cup \mathsf{SP}_1\cup \mathsf{SP}_2$.
\end{enumerate}
\end{theorem} 

We are now ready to state the first main result of this section stating that equilibria of \eqref{eq:motion-abstract} are either normally stable or normally hyperbolic in the sense of Definitions \ref{def:normally-stable} and \ref{def:normally-hyperbolic}. Let $\mathbb{S}^2$ denote the unit sphere in $\R^3$. The following theorem holds.  

\begin{theorem}\label{th:steady-states-normally-stable-hyperbolic}

    \begin{enumerate}[label=\rm{(\alph*)},ref=(\alph*)]
    \item\label{th:normally-sh-SP1}  
    $\Bar{s}=(0,\alpha q,q)\in \mathsf{SP}_1$, with $\alpha \in \R\setminus\{0\}$ and $q\in\mathbb{S}^2$, is either normally stable or normally hyperbolic provided that
    \begin{equation}\label{eq:condition-beta-alpha-lambda-sp1}
        \frac{3\beta^{4}}{\alpha^{4} \lambda_1(\lambda_3 - \lambda_1)} \neq 1.
    \end{equation}
    \item\label{th:normally-sh-SP2}
    $\Bar{s}=(0,\alpha q,q)\in \mathsf{SP}_2$, with $\alpha \in \R\setminus\{0\}$ and $q\in\mathbb{S}^2$, is either normally stable or normally hyperbolic provided that
    \begin{equation}\label{eq:condition-beta-alpha-lambda-sp2}
        \frac{3\beta^{4}}{\alpha^{4} \lambda_2(\lambda_3 - \lambda_2)} \neq 1.
    \end{equation}
    \item\label{th:normally-sh-PR} 
    $\Bar{s}=(0,\alpha q,q)\in \mathsf{PR}$, with $\alpha\in \R\setminus\{0\}$ and $q\in \mathbb{S}^2$,  is either normally stable or normally hyperbolic provided that
    \begin{equation}\label{eq:condition-beta-alpha-lambda-pr}
        \frac{\alpha^2}{\beta^2}(\lambda_3-\lambda_1)\neq 1\qquad\text{and}\qquad  \frac{\alpha^2}{\beta^2}(\lambda_3-\lambda_2) \neq 1.
    \end{equation}
    \end{enumerate}
\end{theorem}
We impose the condition $q\in\mathbb{S}^2$ because, from the physical point of view, $q$ represents the direction of the gravity in the principal frame of inertia. Also, we are interested in general steady states occurring with nonzero angular velocity $\alpha$. 

We note that part \ref{th:normally-sh-PR} of the above theorem has already been proved in \cite[Theorem 4.5]{GaldiMazzone21}. We only prove part \ref{th:normally-sh-SP1} as the proof of part \ref{th:normally-sh-SP2} follows with identical arguments. 
With Propositions \ref{0ineigenvalue}, \ref{prop:0-semi-simple} and \ref{prop:no-imaginary-eigen}, we check that properties (II)--(IV) of Definition \ref{def:normally-stable} hold, all the others follow immediately. 

Let $\bar u=(0,\alpha q,q)\in \mathcal E$, for some $\alpha\in \R$. Consider the operator 
\begin{equation}\label{eq:linear-around-steady}
L:\; u\in D(L):=X_1\mapsto Lu:=Au-F'(\bar u)u, 
\end{equation}
defined in Remark \ref{rm:stability-0}. More explicitly, we have that $L=E^{-1}\bar L$, where 
\begin{equation}\label{eq:bar-L}
\begin{split}
\bar L:\; u&=(v,\omega,z)\in D(\bar L)=X_1
\\
&\mapsto (\P(2\alpha q\times x-\nu\Delta v),
\alpha q \times\I\cdot(\omega-\a) + \omega \times \I\cdot \alpha q-\beta^2\e_3\times z,
\\
&\qquad\qquad\qquad\qquad\qquad\qquad\qquad\qquad\qquad\qquad
\alpha q\times z + \omega \times q)\in X. 
\end{split}
\end{equation}
We may refer to $L$ as the {\em linearized operator around the steady-state solution $\bar u$}.

\begin{proposition} \label{0ineigenvalue}
    Let $\Bar{s}=(0,\alpha q,q)\in \mathsf{SP}_1$, with $\alpha \in \R\setminus\{0\}$ and $q\in\mathbb{S}^2$. Then, 
    \( \dim(N(L)) = 2 \) and $T_{\bar u}\mathcal E=N(L)$.
\end{proposition}
\begin{proof}
Since $T_{\bar u}\mathcal E\subset N(L)$ and $\dim(T_{\bar u}\mathcal E)=2$, it is enough to prove that also $\dim(N(L))=2$. Let $u=(v,\omega,z)\in N(L)$. The equation \( Lu = 0 \) is equivalent to the system:
    \begin{equation}\label{Lieq:motion}
       \left\{
\begin{aligned}
   & \div v=0\qquad&&\text{in }\Omega,
   \\
   &2\alpha q \times v=\nu\Delta v-\nabla p \qquad&&\text{in }\Omega,
   \\
   &\alpha q \times\I\cdot(\omega-\a) + \omega \times \I\cdot \alpha q=\beta^2\e_3\times z,
   \\
   &\alpha q\times z + \omega \times q =0,
   \\
   &v=0\qquad&&\text{on }\partial \Omega. 
\end{aligned}
\right.
    \end{equation}
    
    By taking the dot product of equation \eqref{Lieq:motion}$_2$ with \( v \), integrating over \( \Omega \), and using the boundary condition \eqref{Lieq:motion}$_5$, we obtain
    \[
    \int_\Omega \left( \nu \| \nabla v \|_2^2 + p \, \div v \right) dx = 0.
    \]
    Using the divergence-free condition \eqref{Lieq:motion}$_1$ and the Poincar\'e inequality, we conclude that \( v = 0 \). Consequently, equations \eqref{Lieq:motion}$_3$ and \eqref{Lieq:motion}$_4$ reduce to
    \begin{equation}\label{eq:omega-z-1}
        \alpha q \times \mathbb{I} \cdot \omega + \alpha \omega \times \mathbb{I} \cdot q = \beta^2 \e_3 \times z,
    \end{equation}
    \begin{equation} \label{eq:omega-z-2}
        (-\alpha z  + \omega) \times q = 0.
    \end{equation}
    From \eqref{eq:omega-z-2}, we deduce that
    \[
    \omega = \alpha z + \mu q \quad \text{for some } \mu \in \mathbb{R}.
    \]
    Substituting this into \eqref{eq:omega-z-1}, we obtain
    \begin{equation} \label{eq:main}
        \alpha^2 (q \times \mathbb{I} \cdot z + z \times \mathbb{I} \cdot q) + 2 \alpha \mu (q \times \mathbb{I} \cdot q) = \beta^2 \e_3 \times z.
    \end{equation}
The above algebraic system admits a solution $z=z_i\e_i$ satisfying 
    \[
    z_1= t \text{ (a free parameter),}\quad z_2 = 0, \quad z_3 = -\frac{2\mu}{\alpha} q_3. 
    \]
    Therefore, the kernel of \( L \) is spanned by the following two linearly independent vectors:
    \begin{equation}\label{eq:N(L)-spanned}
    N(L) = \text{span} \left\{ 
        \left(v=0,\omega=q_1\e_1-q_3\e_3,z=-\frac{2}{\alpha} q_3\e_3\right),\;
        \left(v=0,\omega=\alpha\e_1,z=\e_1 \right)
    \right\}.
    \end{equation}
\end{proof}
Our next goal is to show that $\lambda=0$ is a semi-simple eigenvalue of $L$. 

\begin{proposition}\label{prop:0-semi-simple}
If condition \eqref{eq:condition-beta-alpha-lambda-sp1} holds, then $\lambda=0$ is a semi-simple eigenvalue of $L$.
\end{proposition}
\begin{proof}
The interpolation inequality \eqref{eq:interpolation-X_alpha} implies that for all $\varepsilon>0$, there exists a constant $c_\varepsilon>0$ such that for all $u\in X_\alpha$
\[
\norm{F'(\bar u)u}_0\le c_\varepsilon\norm{u}_0+\varepsilon\norm{Au}_0.
\]
By the properties of $A$ outlined in Proposition \ref{prop:A-satisfies-H1}, it follows that
\[
\norm{(\lambda+A)^{-1}}_0\le \frac{M}{\lambda}\quad \text{for all }\lambda>0,
\]
and some constant $M>0$. We then apply \cite[Chapter 4, Section 3, Theorem 3.17]{kato} to state that $L$ has a compact resolvent and, thus, a discrete spectrum. In particular, $\lambda=0$ is an isolated eigenvalue. It remains to show that $X= N(L) \oplus R(L)$. 

Since the operator \( A \) in \eqref{eq:motion-abstract} is closed and invertible, it is also a Fredholm operator of index zero. In addition, $A^{-1}$ is compact and $F\in C^1(X_\alpha;X)$. So, $-F'(\bar u)$ is $A$-compact, and by \cite[Chapter 4, Section 5, Theorem 5.26]{kato}, 
the operator
$L$ defined in \eqref{eq:linear-around-steady} is also Fredholm of index zero. To conclude that $X= N(L) \oplus R(L)$, it then suffices to show that \( N(L) \cap R(L) = \{0\} \).

Let \( (0, \omega_0, z_0) \in N(L)\cap R(L) \). Then, there exists \( (v, \omega, z) \in D(L)=D(A)=X_1 \) satisfying 
    \begin{equation} \label{Lieq:motion1}
        \left\{
        \begin{aligned}
            &\div v = 0 && \text{in } \Omega, \\
            &-\nu \Delta v + \nabla p + 2\alpha q \times v = \omega_0 \times x && \text{in } \Omega, \\
            &\alpha q \times \mathbb{I} \cdot (\omega - \a) + \omega \times \mathbb{I} \cdot  (\alpha q) = \mathbb{I} \cdot \omega_0 + \beta^2 \e_3 \times z, \\
            &\alpha q \times z + \omega \times q = z_0, \\
            &v = 0 && \text{on } \partial \Omega.
        \end{aligned}
        \right.
    \end{equation}
    Our goal is to show that \( \omega_0 = z_0 = 0 \).

    From \eqref{eq:N(L)-spanned}, it follows that
    \[\left.\begin{aligned}
        z_0 &= t_0\e_1-\frac{2\mu_0}{\alpha} q_3 \e_3
        \\
        \omega_0 &= \alpha z_0 + \mu_0 q\
    \end{aligned}\right\}\quad \text{for some } t_0, \mu_0 \in \mathbb{R}. \]
    Taking the dot product of equation \eqref{Lieq:motion1}$_4$ with \( q \), we obtain:
    \[
        z_0 \cdot q = 0,
    \]
    from which it follows that 
    \begin{equation}\label{eq:t0-ker-range}
        t_0 = \frac{2\mu_0}{q_1 \alpha} q_3^2
    \end{equation}
    and
    \begin{equation}\label{eq:omega0-ker-range}
        \omega_0 = \left( \frac{2\mu_0}{q_1} q_3^2 + \mu_0 q_1\right)\e_1-\mu_0 q_3\e_3.
    \end{equation}
    Taking the dot product of equation \eqref{Lieq:motion1}$_3$ with \( \alpha q \), we get
    \begin{equation} \label{prop_eq}
        \alpha^2 (\omega \times \mathbb{I} \cdot q) \cdot q = \alpha \beta^2 (\e_3 \times z) \cdot q + \alpha q\cdot \mathbb{I} \cdot \omega_0.
    \end{equation}
    Moreover, from \eqref{eq:motion}, it follows that the steady-state solution $\bar s=(0,\alpha q,q)$ must satisfy the identity
    \begin{equation} \label{eq:identity}
        \alpha^2 q \times \mathbb{I} \cdot  q = \beta^2 \e_3 \times q.
    \end{equation}
    Substituting \eqref{eq:identity} into \eqref{prop_eq}, we find
    \[
        \beta^2 [(\omega - \alpha z) \times q] \cdot \e_3 = \alpha q\cdot \mathbb{I} \cdot \omega_0.
    \]
    The latter and \eqref{Lieq:motion1}$_{4}$ imply that  
\begin{equation}\label{3.14}
    \beta^{2} z_{0} \cdot \e_{3} = \alpha q\cdot\I \cdot \omega_{0}.
\end{equation}
    Substituting \eqref{eq:t0-ker-range} and \eqref{eq:omega0-ker-range} in the latter displayed equation, we obtain
    \[
        -\frac{2\beta^2 \mu_0}{\alpha^2} q_3=\lambda_1 q_1 \left( \frac{2\mu_0}{q_1} q_3^2 + \mu_0 q_1 \right) + \lambda_3 q_3 (-\mu_0 q_3).
    \]
    Since \( q_1^2 + q_3^2 = 1 \), and using the characterization in \eqref{eq:def-SP1} for $q_3$, we find that
    \[
        \mu_0\left(\lambda_1-\frac{3\beta^4}{\alpha^4 (\lambda_3 - \lambda_1)}\right)=0.
    \] 
   Condition \eqref{eq:condition-beta-alpha-lambda-sp1} immediately implies that \( \mu_0 = 0 \). This information replaced in \eqref{eq:t0-ker-range} and \eqref{eq:omega0-ker-range} yields \( z_0 =  \omega_0 = 0 \). 
\end{proof}

In order to conclude the proof of Theorem \ref{th:steady-states-normally-stable-hyperbolic}, part \ref{th:normally-sh-SP1}, it remains to show that the spectrum of \( L \) does not contain any purely imaginary points (other than $\lambda=0$). To establish this result, we first prove the following lemma.

\begin{lemma} 
    Let $\Bar{s}=(0,\alpha q,q)\in \mathsf{SP}_1$, with $\alpha \in \R\setminus\{0\}$ and $q\in\mathbb{S}^2$, and consider the functional
    \begin{equation} \label{3.16}
        \mathcal{G}_\delta(u) = 
        \mathcal{G}_\delta(v, \omega, z) := \rho\|v\|_{L^2(\Omega)}^2 - \a \cdot \mathbb{I} \cdot \a + \omega_* \cdot \mathbb{I} \omega_* + \delta |z|^2 - 2\alpha z \cdot \mathbb{I} \omega_*,
    \end{equation}
    where \( \omega_* := \omega - \a \). Then there exists \( \hat \delta \in \mathbb{R} \) such that 
    \begin{equation} \label{3.17}
        \frac{1}{2} \frac{d}{dt} \mathcal{G}_{\hat \delta}(u) + \mu \|\nabla v\|_{L^2(\Omega)}^2 = 0
    \end{equation}
    along solutions $u$ of the linearized problem
    \begin{equation} \label{eq:linear-ODE}
        \begin{cases}
            \displaystyle \frac{du}{dt} + Lu = 0, & t > 0, \\
            \displaystyle u(0) \in X
        \end{cases}
    \end{equation}
    where $L$ is the linearized operator defined in \eqref{eq:linear-around-steady}-\eqref{eq:bar-L}. 
\end{lemma}
\begin{proof}
We start by noticing that $L=A-F'(\bar u)$ is a sectorial operator because $A$ is sectorial (by Proposition \ref{prop:A-satisfies-H1}) and $F\in C^1(X_\alpha;X)$ (see Remark \ref{rm:stability-0}). 
Thus, \eqref{eq:linear-ODE} admits a unique solution 
\[
u\in C((0,\infty);X_\alpha))\cap C((0,\infty),X_1)\cap C^1((0,\infty);X).
\]
Writing $u=(v,\omega,z)$, we see that equation \eqref{eq:linear-ODE} can be equivalently rewritten as the following system: 
\begin{equation} \label{3.19}
\left\{
\begin{aligned}
    &\div v = 0 &&\text{in } \Omega \times (0,\infty), \\
    &\partial_t v + (\dot{\omega}_* + \dot{\a}) \times x + 2\omega \times v = \nu \Delta v - \nabla p &&\text{in } \Omega \times (0,\infty), \\
    &\mathbb{I} \cdot \dot{\omega}_* + \alpha q \times \mathbb{I} \cdot \omega_* + \alpha(\omega_* + \a) \times \mathbb{I} \cdot q = \beta^2 \e_3 \times z &&\text{in } (0,\infty), \\
    &\dot{z} + (\omega_* + \a) \times q + \alpha q \times z = 0 &&\text{in } (0,\infty), \\
    &v = 0 &&\text{on } \partial\Omega \times (0,\infty),
\end{aligned}
\right.
\end{equation}
where \( \omega_* = \omega - \a \). 

We begin by taking the dot product of \eqref{3.19}$_2$ with \( \rho v \), and integrate the resulting equation over \( \Omega \). Using integration by parts with the boundary condition \eqref{3.19}$_5$,
it yields
\begin{equation} \label{3.20}
    \frac{1}{2} \left( \frac{d}{dt} \|\rho^{1/2}v\|_{L^2(\Omega)}^2 - \a \cdot \mathbb{I}\cdot \a \right) - \dot{\omega}_* \cdot \mathbb{I} \cdot \a + \mu \|\nabla v\|_{L^2(\Omega)}^2 = 0.
\end{equation}

Next, we take the dot product of \eqref{3.19}$_3$ with \( z \), and of \eqref{3.19}$_4$ with \( \mathbb{I} \cdot \omega_* \), and add the resulting equations. This gives
\begin{equation} \label{3.21}
    \frac{d}{dt} (z \cdot \mathbb{I} \cdot \omega_*) + [(\omega_* + \a) \times q] \cdot (\mathbb{I} \cdot \omega_*) + \alpha[(\omega_* + \a) \times \mathbb{I} \cdot q] \cdot z = 0.
\end{equation}

Taking the dot product of \eqref{3.19}$_3$ with \( \omega_* + \a \), we obtain:
\begin{equation} \label{3.22}
    \frac{1}{2} \frac{d}{dt} (\omega_* \cdot \mathbb{I} \cdot \omega_*) + \alpha (q \times \mathbb{I} \cdot \omega_*) \cdot (\omega_* + \a) + \a \cdot \mathbb{I} \cdot \dot{\omega}_* - \beta^2 (\e_3 \times z) \cdot (\omega_* + \a) = 0.
\end{equation}

Finally, taking the dot product of \eqref{3.19}$_4$ with \( z \), we obtain:
\begin{equation} \label{3.23}
    \frac{1}{2} \frac{d}{dt} |z|^2 + [(\omega_* + \a) \times q ]\cdot z = 0.
\end{equation}

Now we consider the linear combination:
\eqref{3.20} + \eqref{3.22} + $\delta$ \eqref{3.23} - $\alpha$ \eqref{3.21} and deduce that 
\begin{equation}
    \frac{1}{2} \frac{d}{dt} \mathcal{G}_\delta + \mu \|\nabla v\|_{L^2(\Omega)}^2 = (\omega_* + \a) \cdot \left( \alpha^2 \mathbb{I} \cdot q + \beta^2 \e_3 - \delta q \right) \times z,
\end{equation}
with $\mathcal G_\delta$ defined in equation \eqref{3.16}. Finally, using \eqref{eq:identity}, 
we find \(\hat \delta \in \mathbb{R} \) such that 
\[
    \alpha^2 \mathbb{I} \cdot q + \beta^2 \e_3 = \hat \delta q,
\]
and this concludes the proof.
\end{proof}
Now, we are ready to prove that \( \sigma(L) \cap i\mathbb{R} = \{0\} \).

\begin{proposition}\label{prop:no-imaginary-eigen}
 Let $\Bar{s}=(0,\alpha q,q)\in \mathsf{SP}_1$, with $\alpha \in \R\setminus\{0\}$ and $q\in\mathbb{S}^2$, and consider the linearized operator $L$ defined in \eqref{eq:linear-around-steady}-\eqref{eq:bar-L}. 
 Then, 
     \begin{align*}
         \sigma(L) \cap i\mathbb{R} = \{0\}.
     \end{align*}
\end{proposition}
\begin{proof}
    To prove the stated property, we show that there is no bounded periodic solution of \eqref{eq:linear-ODE} other than the trivial one. We argue by contradiction, and assume that system \eqref{3.19} admits a time-periodic solution \((v, \omega_{*}, z)\), with period \(T > 0\), such that
    \begin{equation}\label{3.25}
        \int^{T}_{0} v(x,t) \, dt = \int^{T}_{0} \omega_{*}(x,t) \, dt = \int^{T}_{0} z(x,t) \, dt = 0.
    \end{equation}
    Integrating \eqref{3.17} from \(0\) to \(T\) and using periodicity, we deduce that 
    \begin{equation}\nonumber
        \int^{T}_{0} \lVert \nabla v \rVert_{L^2(\Omega)}^2 \, dt = 0.
    \end{equation}
    This implies that \(v \equiv 0\) by \eqref{3.19}$_{5}$ and Poincar\'e inequality. In particular, $\a=0$. Using this information in \eqref{3.19}$_{2}$, we have that 
    \[
    \dot{\omega}_{*} \times x = - \nabla p.
    \]
    Applying the curl operator on both sides of the latter equation, we obtain that \(\dot{\omega}_{*} = 0\), that is, \(\omega_{*}\) is constant. Equation \eqref{3.25} then implies that \(\omega_{*} = 0\). The latter, in conjunction with \eqref{3.19}$_{3}$, gives \(z = z_{3} \e_3\). On the other side, from \eqref{3.19}$_{4}$, we also have that \(q \cdot \dot{z} = 0\). Therefore, \(\dot{z}_{3} = 0\), which allows us to conclude that \(z_{3}\) is constant, and then zero by \eqref{3.25}.
\end{proof}
\begin{remark}
    It is possible to provide a more precise characterization of the stability properties of the steady states in Theorem \ref{th:steady-states} in terms of physical quantities like the mass distribution within $\mathscr{S}$. We avoid this here because we are mainly interested in the long-time behaviour of generic trajectories, and refer the interested reader to \cite{Dutta-thesis}. 
\end{remark}

Concerning the time-dependent solutions, \cite{GaldiMazzoneMohebbi18} provides the state of the art on the existence of global in time solutions to \eqref{eq:motion}. Let us start by recalling the definition of weak solutions.

The weak formulation of problem~\eqref{eq:motion} is derived in the standard manner: by taking the dot product of equation~\eqref{eq:motion}$_1$ with a test function \( \psi \in {_0}H^1_{2,\sigma}(\Omega)\), and then integrating (by parts) the resulting expression over the space-time domain \( \Omega \times (0, t) \). This procedure yields the following equations:

\begin{equation} \label{Weak_form_velocity}
\begin{aligned}
( v(t), \psi) &+  \omega(t) \cdot \int_\Omega x \times \psi \, dx \\
&+ \int_0^t \left\{ (v_i\partial_{x_i}v, \psi) 
+ 2 ((\omega \times v, \psi) + \nu (\nabla v, \nabla \psi) \right\} \\
&= ( v(0), \psi) +  \omega(0) \cdot \int_\Omega x \times \psi \, dx, \mbox{ for all } t \in (0, \infty).
\end{aligned}
\end{equation}

Moreover, integrating equations \eqref{eq:motion}$_3$ and \eqref{eq:motion}$_4$ over \( (0, t) \), we obtain:

\begin{multline} \label{Weak_form_angular}
\I \cdot (\omega(t)- \a(t)) = \I \cdot (\omega(0)-\a(0)) 
\\
- \int_0^t \left[ \omega \times (I \cdot (\omega-\a) - \beta^2 (\e_3 \times \gamma) \right], \mbox{ for all } t \in (0, \infty),
\end{multline}
and
\begin{equation} \label{Weak_form_gravity}
\gamma(t) = \gamma(0) - \int_0^t \omega \times \gamma \, , \mbox{ for all } t \in (0, \infty).
\end{equation}

\begin{definition}\label{def:weak-solutions}
    The triplet $(v, \omega, \gamma)$ is a weak solution to \eqref{eq:motion} if it meets the following requirements:
    \begin{enumerate}[label=\rm{(\alph*)},ref=(\alph*)]
    \item \( v \in C_w([0, \infty); L^2_\sigma(\Omega)) \cap L^\infty(0, \infty;  L^2_\sigma(\Omega)) \cap L^2(0, \infty; {_0}H^1_{2,\sigma}(\Omega))\). 
    \item $\omega \in C^0([0, \infty);\R^3) \cap C^1((0, \infty);\R^3)$.
    \item $\gamma \in C^0([0, \infty);\mathbb{S}^2)\cap C^1((0, \infty); \mathbb{S}^2)$. 
    \item $(v, \omega, \gamma)$ satisfies the {\em strong energy inequality} for all \( t \geq s \) and for almost all \( s \geq 0 \), including \( s = 0 \),
    \begin{equation}\label{eq:strong-energy}
    \mathsf{E}(t) + \mathsf{U}(t) + 2\nu \int_s^t \| \nabla v(\tau) \|^2_{L^2(\Omega)} \; d\tau \leq \mathsf{E}(s) + \mathsf{U}(s),\end{equation}
    where,
    \[
    \mathsf{E}(t) := \rho\norm{v(t)}^2_{L^2(\Omega)} -\a\cdot\I\cdot\a+ (\omega(t)-\a)\cdot \I \cdot (\omega(t)-\a),
    \] 
    and $\mathsf{U}(t) := -2\beta^2 \gamma(t) \cdot \e_3$ .
\item The triplet \( (v, \omega, \gamma) \) satisfies equations~\eqref{Weak_form_velocity}, \eqref{Weak_form_angular}, and~\eqref{Weak_form_gravity}.
\end{enumerate}
\end{definition}

We note that the functional 
\begin{multline*}
(v_1,v_2)\in L^2(\Omega)\times L^2(\Omega)
\\
\mapsto \int_{\Omega}\rho v_1\cdot v_2\; \d x
-\left(\int_{\Omega}\rho x\times v_1\; dx\right)\cdot\I\cdot \left(\int_{\Omega}\rho x\times v_2\; dx\right)\in \R
\end{multline*}
defines an inner product in $L^2(\Omega)$ (see \cite[Chapter 1,
Section 7.2.2 and 7.2.3]{KoKr} for a proof). The associated norm
\begin{equation*}
\begin{split}
\norm{\cdot}:\;v\in L^2(\Omega)\mapsto \norm{v}&:=\rho \norm{v}^2_{L^2(\Omega)} -\left(\int_{\Omega}\rho x\times v\; dx\right)\cdot\I\cdot \left(\int_{\Omega}\rho x\times v\; dx\right)
\\
&\quad \equiv\rho \norm{v}^2_{L^2(\Omega)} -\a \cdot \I \cdot \a
\end{split}
\end{equation*}
is equivalent to $\norm{\cdot}_{L^2(\Omega)}$. 
Specifically, there exists a constant $c\in (0,1)$ such that 
\begin{equation}\label{eq:equivalent-norm}
c\rho\norm{v}_{L^2(\Omega)}\le \norm{v}\le \rho\norm{v}_{L^2(\Omega)}. 
\end{equation}
In view of these considerations, we note that 
\begin{equation}\label{eq:equivalent-energy}
\begin{aligned}
 \mathsf{E}=\norm{v}+(\omega-\a)\cdot\I\cdot (\omega-\a). 
\end{aligned}
\end{equation}

The class of weak solutions in the sense of Definition \ref{def:weak-solutions} is nonempty as shown in \cite[Proposition 5.2.]{GaldiMazzoneMohebbi18}. Furthermore, such solutions enjoy the following ``regularization'' properties. 

\begin{proposition}\label{prop:weak-strong}
Let \( u:= (v, \omega, \gamma) \) be a weak solution corresponding to initial data in $L^2_\sigma(\Omega) \times \R^3\times \R^3$. Then, there exists \( t_0 = t_0(u) > 0 \) such that for all \( T > 0 \):
\begin{equation}\label{eq:regularity}
\begin{split}
&v \in C^0([t_0, t_0 + T]; {_{0}H}^1_{2}(\Omega)) \cap L^2(t_0, t_0 + T; {_{0}H}^2_{2,\sigma}(\Omega))
\cap H^1_2(t_0,t_0+T;L^2_\sigma(\Omega)),
\\
&\omega\in C^1([t_0, t_0 + T];\R^3), \quad \gamma \in C^1([t_0, t_0 + T];\mathbb{S}^2).
\end{split}
\end{equation}

Moreover, there exists \( p \in L^2(t_0, t_0 + T; H^1_2(\Omega)) \), for all \( T > 0 \), such that the quadruple \( (v, \omega, \gamma, p) \) satisfies equation \eqref{eq:motion} almost everywhere in \( (t_0, \infty) \). Finally,
\begin{align}
\label{eq:vto0-nabla}
&\lim_{t \to \infty} \norm{v(t)}_{H^1_2(\Omega)}=0
\\
& \label{eq:vto0-partial-t}
\lim_{t\to\infty}\norm{\partial_t v(t)}_{L^2(\Omega)}= 0.
\end{align}
\end{proposition}

\begin{proof}
Properties \eqref{eq:regularity} and \eqref{eq:vto0-nabla} have been proved in \cite[Proposition 5.5.]{GaldiMazzoneMohebbi18}. It remains to show \eqref{eq:vto0-partial-t}. 

Without loss of generality, we assume that $v(t_0)\in {_{0}H}^2_{2,\sigma}(\Omega)$ and 
\begin{equation}\label{eq:uniform-H12-v}
\norm{v}_{H^1_2(\Omega)}\le K\qquad\text{ for all }t\ge t_0,
\end{equation}
and some positive constant $K$. A smooth fluid velocity $v$ in \eqref{eq:motion} (or its smooth approximation in the Faedo-Galerkin method starting at $v(t_0)$) satisfies the estimate \cite[Equation (5.21)]{GaldiMazzoneMohebbi18}, that is, 
\begin{equation}\label{eq:ode-nabla-v}
\frac{\d}{\d t}\norm{\nabla v}^2_{L^2(\Omega)}+c_1\norm{\partial_t v}^2_{L^2(\Omega)}+c_2\norm{v}^2_{H^2_2(\Omega)}\le c_3\left(\norm{\nabla v}^2_{L^2(\Omega)} + 1\right).
\end{equation}

In addition, the quadruple $(w:=\partial_t v,\xi:=\dot\omega,\gamma,\hat p=\partial_t p)$ satisfies the following equations 
\begin{equation}\label{eq:motion-differentiated}
\left\{
\begin{aligned}
   &\div w=0\qquad&&\text{in }\Omega\times(t_0,t_0+T),
   \\
   &\partial_t w+\dot\xi\times x+w_i\partial_{x_i}v+v_i\partial_{x_i}w+2\xi\times v \qquad&&\text{in }\Omega\times(t_0,t_0+T),
   \\
   &\qquad\qquad\qquad\quad 
   +2\omega\times w=\nu\Delta w-\nabla \hat p-g\omega\times\gamma \qquad&&
   \\
   &\I\cdot(\dot\xi-\dot \a_w)+\xi\times\I\cdot(\omega-\a_w) \qquad &&\text{in }(t_0,t_0+T),
   \\
   &\qquad\qquad\qquad\quad  
   +\omega\times\I\cdot(\xi-\a_w)=-\beta^2\e_3\times \left(\omega\times\gamma\right) \qquad 
   \\
   &w=0\qquad&&\text{on }\partial \Omega\times (t_0,t_0+T),
   \\
   &a_w:=-\I^{-1}\cdot\left(\int_{\Omega}\rho x\times w\;\d x\right)&&\text{in }(t_0,t_0+T),
\end{aligned}
\right.    
\end{equation}
obtained by (formally) taking the time derivative in \eqref{eq:motion}$_{1,2,3,5}$. Let us (again, formally) take the dot product of \eqref{eq:motion-differentiated}$_2$ with $\rho w$ and integrate the resulting equation over $\Omega$, using integration by parts and \eqref{eq:motion-differentiated}$_{1,4}$, we find that 
\[
\frac\rho2 \frac{\d }{\d t}\norm{w}^2_{L^2(\Omega)}-\a_w\cdot\I\cdot\dot\xi
+\int_{\Omega}\left(w_i\partial_{x_i}v\right)\cdot w\; \d x
+2\int_{\Omega}\rho (\xi\times v)\cdot w\;\d x
=-\mu\norm{\nabla w}^2_{L^2(\Omega)}.
\]
Next, we take the dot product of \eqref{eq:motion-differentiated}$_3$ by $\a_w$ and we see that
\[
\a_w\cdot \I\cdot\dot\xi=\frac 12 \frac{\d}{\d t}(\a_w\cdot\I\cdot\a_w)
-[\xi\times\I\cdot(\omega-\a_w)+\omega\times\I\cdot(\xi-\a_w)+\beta^2\e_3\times \left(\omega\times\gamma\right)]\cdot\a_w.
\]
Adding side by side the latter displayed equations, we obtain 
\begin{equation}\label{eq:energy-partial-t-v}
\frac 12\frac{\d }{\d t}
\norm{w}^2+\mu\norm{\nabla w}^2_{L^2(\Omega)}=\mathsf{F},
\end{equation}
where
 
\begin{multline}\label{eq:F-energy-partial-t-v}
\mathsf{F}(t):=-\int_{\Omega}\left(w_i\partial_{x_i}v\right)\cdot w\; \d x
-2\int_{\Omega}\rho (\xi\times v)\cdot w\;\d x
\\
-[\xi\times\I\cdot(\omega-\a_w)+\omega\times\I\cdot(\xi-\a_w)+\beta^2\e_3\times \left(\omega\times\gamma\right)]\cdot\a_w. 
\end{multline}
Let us estimate each term in $\mathsf F$. Applying 
H\"older's inequality, interpolation inequality, Sobolev embedding theorem in conjunction with Poincar\'e inequality, and Young's inequality, we observe that
\[\begin{split}
\left|\int_{\Omega}\left(w_i\partial_{x_i}v\right)\cdot w\; \d x\right|
&\le \norm{w}^2_{L^4(\Omega)}\norm{\nabla v}_{L^2(\Omega)}
\le c_4\norm{w}^{1/2}_{L^2(\Omega)}\norm{\nabla w}^{3/2}_{L^2(\Omega)}\norm{\nabla v}_{L^2(\Omega)}
\\
&\le \frac{\mu}{2}\norm{\nabla w}^{2}_{L^2(\Omega)}+c_5\left(\norm{w}^{6}_{L^2(\Omega)}+\norm{\nabla v}^6_{L^2(\Omega)}\right). 
\end{split}\]
For the remaining terms, we use \eqref{eq:motion}$_3$, the fact that $\gamma\in \mathbb{S}^2$, the strong energy inequality \eqref{eq:strong-energy} and \eqref{eq:equivalent-norm}, to conclude that 
\[
|\mathsf{F}(t)|\le \frac{\mu}{2}\norm{\nabla w}^{2}_{L^2(\Omega)} +c_6
\left(\norm{w}^{6}+\norm{\nabla v}^6_{L^2(\Omega)} + 1\right). 
\]
So, $w$ satisfies the following energy estimate
\begin{equation}\label{eq:ode-time-derivative}
\frac{\d }{\d t}\norm{w}^2+\mu\norm{\nabla w}^2_{L^2(\Omega)}\le 
2c_6\left(\norm{w}^{6}+\norm{\nabla v}^6_{L^2(\Omega)} + 1\right). 
\end{equation}
Set $\mathsf{E}_1:=\norm{w}^2+\norm{\nabla v}^2_{L^2(\Omega)}$. Adding \eqref{eq:ode-nabla-v} and \eqref{eq:ode-time-derivative} side by side and using \eqref{eq:equivalent-norm}, we obtain that $\mathsf{E}_1$ enjoys the following differential inequality:
\begin{equation}\label{eq:ode-nabla-partial-t-v}
\frac{\d \mathsf{E}_1}{\d t}
+c_7\left(\norm{w}^2_{L^2(\Omega)}
+\norm{\nabla w}^2_{L^2(\Omega)}
+\norm{v}^2_{H^2_2(\Omega)}\right)\le c_8\mathsf{E}_1^3.
\end{equation}
Integrating, we find that there exists a time $t_*$, with 
\[
t_*\ge \frac{c_9}{\norm{\nabla(v(t_0))}^4_{L^2(\Omega)}+\norm{w(t_0)}^4_{L^2(\Omega)}+1} 
\]
and $c_9$ independent of $t_0$, and functions $f_i\in C([t_0,t_0+t_*);[0,\infty))$, $i=1,2,3$, such that 
\begin{equation}\label{eq:estimate-nabla-partial-t-v}\begin{split}
&\norm{v}_{H^1_2(\Omega)}\le f_1(t),\qquad \norm{w}_{L^2(\Omega)}\le f_2(t)
\\
&\int_{t_0}^t\left(\norm{w}^2_{L^2(\Omega)}
+\norm{\nabla w}^2_{L^2(\Omega)}+\norm{v}^2_{H^2_2(\Omega)}\right)\;\d \tau\le f_3(t)
\quad\text{in }[t_0,t_0+t_*). 
\end{split}\end{equation}
We next claim that $w(t_0)\in L^2_\sigma(\Omega)$ provided that $v(t_0)\in {_{0}H}^2_{2,\sigma}(\Omega)$. Recall that $w=\partial_t v$, then take the dot product of \eqref{eq:motion}$_2$ by $w$ and integrate the resulting equation of $\Omega$. Applying H\"older inequality, \eqref{eq:motion}$_3$ together with \eqref{eq:strong-energy}, we find that 
\begin{equation}\label{eq:initial-regularity-partial-t-v}
\norm{w(t_0)}\le c_{10}\left(\norm{v(t_0)}^3_{H^2_2(\Omega)}+\norm{v(t_0)}_{H^2_2(\Omega)}+1\right).
\end{equation}
In addition, $t_*\ge c_{11}\left(\norm{v(t_0)}^{12}_{H^2_2(\Omega)}+\norm{v(t_0)}^4_{H^2_2(\Omega)}+1\right)^{-1}$. These latter two facts in conjunction with the a priori estimates \eqref{eq:estimate-nabla-partial-t-v} and the Faedo-Galerking method lead to the existence of a weak solution $(\hat v,\hat \omega,\hat \gamma)$ corresponding to the initial data $(v(t_0),\omega(t_0),\gamma(t_0))$, which enjoys the regularity properties \eqref{eq:regularity} and also 
\[
\partial_t \hat v\in L^\infty(t_0,t_0+\tau;L^2_\sigma(\Omega))\cap L^2(t_0,t_0+\tau;{_{0}H}^1_{2,\sigma}(\Omega))\quad \text{for all }\tau\in (0,t_*),
\]
we refer the interested reader to \cite{Ma,Ma2} for the technical details. By the {\em weak-strong uniqueness} proved in \cite[Proposition Proposition 5.4.]{GaldiMazzoneMohebbi18}, we further obtain that the weak solution $(\hat v,\hat \omega,\hat \gamma)$ is indeed strong, and 
$(\hat v=v,\hat \omega=\omega,\hat \gamma=\gamma)$ in $[t_0,t_0+t_*)$. We next claim that $t_*=\infty$. Assuming the contrary, that is $t_*<\infty$, necessarily
\[
\lim_{t\to t_*^-}\norm{v(t)}_{H^2_2(\Omega)}=\infty. 
\]
We show that the latter condition does not hold. To this aim, we go back to \eqref{eq:energy-partial-t-v} and estimate $\mathsf{F}$ in \eqref{eq:F-energy-partial-t-v} as follows:
\[\begin{split}
|\mathsf{F}|&\le \norm{w}_{L^4(\Omega)}\norm{v}_{L^4(\Omega)}\norm{\nabla w}_{L^2(\Omega)}
+k_1(\norm{v}+1)\norm{w}^2_{L^2(\Omega)}+k_2(\norm{v}+1)\norm{w}_{L^2(\Omega)}
\\
&\le \frac{\mu}{2}\norm{\nabla w}^2_{L^2(\Omega)}+k_3(\norm{\nabla v}^8_{L^2(\Omega)}+\norm{v}+1)\norm{w}^2_{L^2(\Omega)}+k_2(\norm{v}+1)\norm{w}_{L^2(\Omega)}.
\end{split}\]
Using the last estimate, \eqref{eq:strong-energy}, \eqref{eq:regularity}, \eqref{eq:uniform-H12-v}, Poincar\'e inequality and \eqref{eq:equivalent-norm} in \eqref{eq:energy-partial-t-v}, we find that 
\begin{equation}\label{eq:energy-partial-t-estimated}
\frac{\d}{\d t}\norm{w}^2+k_4\norm{w}^2_{H^1_2(\Omega)}\le k_5\norm{w}^2_{L^2(\Omega)}+k_6\norm{w}_{L^2(\Omega)}.
\end{equation}
We  integrate
the latter equation and use \eqref{eq:initial-regularity-partial-t-v} to infer that
\[
w=\partial_t v\in L^\infty(t_0,t_0+t_*;L^2_\sigma(\Omega))\cap L^2(t_0,t_0+t_*;{_{0}H}^1_{2,\sigma}(\Omega)).
\]
Furthermore, by dot-multiplying \eqref{eq:motion}$_2$ by $\P\Delta v$ and using \eqref{eq:motion}$_3$ in conjunction with \eqref{eq:strong-energy} (which still holds for $t\ge t_0$), we have that 
\[\begin{split}
\nu\norm{\P\Delta v}^2_{L^2(\Omega)}&\le\int_{\Omega}\left|\left(\partial_t v+\dot \a\times x+(\dot\omega-\dot \a)\times x+v_i\partial_{x_i}v+2\omega\times v\right)\cdot\P\Delta v\right|\; \d x
\\
&\le \frac{\nu}{2}\norm{\P\Delta v}_{L^2(\Omega)}^2+k_4\left(\norm{\partial_t v}^2_{L^2(\Omega)}+\norm{\nabla v}^6_{L^2(\Omega)}+1\right).
\end{split}\]
Next, we use the estimate $\norm{v}_{H^2_2(\Omega)}\le k\norm{\P\Delta v}_{L^2(\Omega)}$ (see
\cite[Theorem IV.6.1]{Galdi}) to conclude that 
\begin{equation}\label{eq:estimate-H22-v}
\norm{v}_{H^2_2(\Omega)}\le k_5\left(\norm{\partial_t v}_{L^2(\Omega)}+\norm{\nabla v}^3_{L^2(\Omega)}+1\right)\quad\text{for a.e. }t\in [t_0,t_0+t_*),
\end{equation}
and necessarily $t_*=\infty$. 

We go back to \eqref{eq:energy-partial-t-estimated} and use the Gr\"onwall-type lemma proved in \cite[Lemma 2.3.4.]{Ma2} to conclude that $\norm{w}\to 0$ as $t\to\infty$, and then \eqref{eq:vto0-partial-t} holds by \eqref{eq:equivalent-norm}. 
\end{proof}

As a corollary of the latter proposition, we have that trajectories corresponding to weak solutions to \eqref{eq:motion} become uniformly bounded after a sufficiently large time, and remain close to the set of steady states $\mathcal E$ (characterized in Theorem \ref{th:steady-states}) in an appropriate topology. To state this fact in mathematical terms, we introduce the space  
\begin{multline*}
    C_b([t_1,+\infty);X_1):=\{f\in C([t_1,+\infty);X_1):\; 
    \\
    \sup_{t\in [t_1,+\infty)}\norm{f(t)}_{X_1}\le M\text{ for some }M>0\},
\end{multline*}
where we recall the definition of $X_1$ given in \eqref{eq:def-X1}. The following result holds. 

\begin{corollary}\label{cor:uniform-bound-trajectories}
Let $p\in [1,\infty)$. Under the hypotheses of Proposition \ref{prop:weak-strong}, it follows that 
\begin{equation}\label{eq:uniform-bound-trajectories}
u=(v, \omega, \gamma)\in C_b([t_1,+\infty);X_1)
\end{equation}
for some $t_1\ge t_0$, where $t_0$ is the time found in Proposition \ref{prop:weak-strong}. 

In addition, 
\begin{equation}\label{eq:omega-limit}
\lim_{t\to+\infty}\dist_{X_\alpha}(u(t),\mathcal E)=0\qquad \text{for any fixed $\alpha\in [0,1)$.}
\end{equation}
Moreover, \eqref{eq:omega-limit} continues to hold with $p\leq2$ and $\alpha =1$.
\end{corollary}

\begin{proof}
Under the hypotheses of Proposition \ref{prop:weak-strong} and following its proof, we have that $v$ enjoys the estimate \eqref{eq:estimate-H22-v} a.e. on $[t_0,\infty)$. Since \eqref{eq:strong-energy},\eqref{eq:vto0-nabla}, and \eqref{eq:vto0-partial-t} hold, we immediately find \eqref{eq:uniform-bound-trajectories} when $q=2$. 

Since $u=(v, \omega, \gamma)\in C_b([t_0,+\infty);H^2_2(\Omega)\times\R^3\times\mathbb{S}^2)$, the Sobolev embedding theorem implies that 
\[
v\in C_b([t_0,+\infty);H^1_6(\Omega))\cap C_b([t_0,+\infty);L^\infty(\Omega)),
\]
and the map $t\in [t_0,+\infty)\mapsto F(u(t))\in L^6(\Omega)\times\R^3\times\R^3$ is uniformly bounded. 

Let $p=6$. By the variation of parameter formula for \eqref{eq:motion-abstract} applied to $u$, there exists $\tau\ge t_0$ such that 
\[
\sup_{t\ge \tau}\norm{u(t)}_{X_\alpha}<+\infty\quad\text{for each }\alpha\in (0,1), 
\]\
where $X_\alpha=[L^6(\Omega),H^2_6(\Omega)]_\alpha\times\R^3\times\mathbb{S}^2$. Now choose $\tilde\alpha\in (3/4,1)$. By the Sobolev embedding theorem, $v\in H^{2\tilde\alpha}_6(\Omega)\hookrightarrow H^1_\infty(\Omega)$. This, in turn, implies that the map $t\in [\tau,+\infty)\mapsto F(u(t))\in L^p(\Omega)\times\R^3\times\R^3$ is uniformly bounded for each $p>6$. Using again the variation of parameter formula for \eqref{eq:motion-abstract} with the same $u=(v, \omega, \gamma)$ above, we find $\bar \tau\ge \tau\ge t_0$ such that  
\[
\sup_{t\ge \bar \tau}\norm{u(t)}_{X_\alpha}<+\infty\quad\text{for each }p>6\text{ and }\alpha\in (0,1). 
\]
The latter fact continues to hold for $p\in [1,6]$ since $\Omega$ is bounded. 

Fix $p\in [1,\infty)$ and $\alpha\in [0,1)$, and choose $\beta\in (0,1)$ and $r\in (1,\infty)$ such that 
\[
2\alpha=1+(2\beta-1)\theta,\qquad \frac{1}{p}=\frac{1-\theta}{2}+\frac{\theta}{r}\quad\text{for some }\theta\in [0,1).
\] 
Therefore, $H^{2\alpha}_p(\Omega)=[H^1_2(\Omega),H^{2\beta}_r(\Omega)]_\theta$ and 
\begin{equation}\label{eq:limit-H^2alpha_q}
\norm{v}_{H^{2\alpha}_p(\Omega)}\le c\norm{v}_{H^{1}_2(\Omega)}^{1-\theta}\norm{v}_{H^{2\beta}_r(\Omega)}^\theta
\le C\norm{v}_{H^{1}_2(\Omega)}^{1-\theta}\to 0\quad\text{as }t\to\infty. 
\end{equation}

Now consider $p=2$. Since the map  
\[
t \in [t_1, +\infty) \mapsto F(u(t)) \in H^1_2(\Omega) \times \mathbb{R}^3 \times \mathbb{R}^3
\]
is uniformly bounded, applying once more the variation of parameters formula to equation~\eqref{eq:motion-abstract} (with the same decomposition \( u = (v, \omega, \gamma) \) as above), there exists a time \(  t_1 \geq \bar \tau\ge \tau \geq t_0 \) such that  
\[
\sup_{t \geq t_1} \| A^{\tilde{\alpha}} u(t) \|_X < +\infty, \quad
\text{ for some } \tilde{\alpha} > 1.
\]
Fix \( \tilde{\theta} \in [0,1] \) such that  
\[
1 = \tilde{\theta} \tilde{\alpha} + (1 - \tilde{\theta}) \tfrac{1}{2}.
\]  
By interpolation\footnote{Since $A$ has bounded imaginary powers, then $[D(A^{\alpha_1}), D(A^{\alpha_2})]_{\tilde \theta}= D(A^{(1-\tilde \theta)\alpha_1+\tilde {\theta} \alpha_2})$, for $\alpha_1= \frac12$ and $\alpha_2=\tilde \alpha$.}, we then deduce
\begin{equation}\label{eq:limit-H^2}
\| v(t) \|_{H^2_2(\Omega)}\leq c_1 \| v(t) \|_{H^{2\tilde\alpha}_2(\Omega)}^{\tilde{\theta}} \| v (t) \|_{H^1_2(\Omega)}^{1 - \tilde{\theta}} 
\leq c_2 \| v (t) \|_{H^1_2(\Omega)}^{1 - \tilde{\theta}}  \to 0 \quad \text{as } t \to \infty.
\end{equation}

Let us now consider the $\Omega$-limit set corresponding to the weak solution $(v,\omega,\gamma)$. 

\cite[Proposition 6.3.]{GaldiMazzoneMohebbi18} provides a complete characterization of such $\Omega$-limit set in terms of steady-states solutions to \eqref{eq:motion}. From this characterization and the limits in \eqref{eq:limit-H^2alpha_q} and \eqref{eq:limit-H^2}\footnote{Recall that the remaining components of $u$ belong to a finite-dimensional space. }, \eqref{eq:omega-limit} immediately follows. 
\end{proof}

We are ready to state and prove our main result on the long-time behaviour of weak solutions.

\begin{theorem}\label{th:main-application}
Consider any initial data $(v_0,\omega_0,\gamma_0)\in L^2_\sigma(\Omega)\times\R^3\times\mathbb{S}^2$ satisfying the conditions
\begin{equation}\label{eq:condition-i.d.}\begin{split}
&\frac{\lambda_3-\lambda_i}{\lambda_3^2}|\gamma_0\cdot\I\cdot(\omega_0-\a_0)|^2\ne\beta^2\quad\text{ for }i=1,2,
\\
&(\lambda_3-\lambda_i)\gamma_0\cdot\I\cdot(\omega_0-\a_0)\ne 4\beta^4 \quad\text{ for }i=1,2.
\end{split}\end{equation}
Then, for each corresponding weak solution $(v, \omega, \gamma)$, there exist $t_\infty,c,k >0$ and $(0, \omega_\infty, \gamma_\infty) \in \mathcal{E}$ such 
   \begin{equation*}
        \begin{aligned}
        \norm{v(t)}_{H^{2\alpha}_p(\Omega)}+ |\omega(t) - \omega_\infty| + |\gamma(t)-\gamma_\infty| \leq c \exp^{-kt}
        \end{aligned}
    \end{equation*}
for all $t>t_\infty$, for any fixed $\alpha\in [0,1)$ and $p\in [1,\infty)$ as well as $\alpha=1$ when $p=2$.
\end{theorem}

\begin{proof}
By \cite[Proposition 6.3]{GaldiMazzoneMohebbi18}, the $\Omega$-limit set corresponding to a weak solution $u=(v,\omega,\gamma)$ to \eqref{eq:motion}, with initial data initial data $(v_0,\omega_0,\gamma_0)\in L^2_\sigma(\Omega)\times\R^3\times\mathbb{S}^2$, is characterized by 
\[
\Omega(u):=\{(0,\alpha q,q)\in \mathcal E:\; \alpha q\cdot\I\cdot q=\gamma_0\cdot\I\cdot(\omega_0-\a_0)\}
\]
which, in turn, implies that $\alpha$ satisfies an algebraic equation with finite degree. Specifically, we have the following cases:
\begin{enumerate}
    \item If $\lambda_1=\lambda_2=\lambda_3$, then
        \[
        \Omega(u)\subset\{(0,\alpha q,q)\in \mathsf{PR}:\; \alpha=\pm\gamma_0\cdot(\omega_0-\a_0)\}.
        \]
    \item If $\lambda_1\ne\lambda_2=\lambda_3$, then
        \begin{multline*}
        \Omega(u)\subset \left\{(0,\alpha q,q)\in \mathsf{PR}:\; \alpha=\pm\frac{\gamma_0\cdot\I\cdot(\omega_0-\a_0)}{\lambda_3}\right\}
        \\
        \cup\left\{(0,\alpha q,q)\in \mathsf{SP}_1:\; (\lambda_3-\lambda_1)\lambda_1\alpha^4-(\lambda_3-\lambda_1)\gamma_0\cdot\I\cdot(\omega_0-\a_0)\alpha^3+\beta^4=0\right\}.
        \end{multline*}
    \item If $\lambda_1=\lambda_3\neq\lambda_2$, then
        \begin{multline*}
        \Omega(u)\subset \left\{(0,\alpha q,q)\in \mathsf{PR}:\; \alpha=\pm\frac{\gamma_0\cdot\I\cdot(\omega_0-\a_0)}{\lambda_3}\right\}
        \\
        \cup\{(0,\alpha q,q)\in \mathsf{SP}_2:\; (\lambda_3-\lambda_2)\lambda_2\alpha^4-(\lambda_3-\lambda_2)\gamma_0\cdot\I\cdot(\omega_0-\a_0)\alpha^3+\beta^4=0\}.
        \end{multline*}
    \item If $\lambda_1\ne \lambda_2\ne \lambda_3$, then 
        \begin{equation*}
        \begin{aligned}
        &\Omega(u)\subset \left\{(0,\alpha q,q)\in \mathsf{PR}:\; \alpha=\pm\frac{\gamma_0\cdot\I\cdot(\omega_0-\a_0)}{\lambda_3}\right\}
        \\
        &\cup\left\{(0,\alpha q,q)\in \mathsf{SP}_1:\; (\lambda_3-\lambda_1)\lambda_1\alpha^4-(\lambda_3-\lambda_1)\gamma_0\cdot\I\cdot(\omega_0-\a_0)\alpha^3+\beta^4=0\right\}
        \\
        &\cup\{(0,\alpha q,q)\in \mathsf{SP}_2:\; (\lambda_3-\lambda_2)\lambda_2\alpha^4-(\lambda_3-\lambda_2)\gamma_0\cdot\I\cdot(\omega_0-\a_0)\alpha^3+\beta^4=0\}.
        \end{aligned}
        \end{equation*}
\end{enumerate}

Thanks to \eqref{eq:condition-i.d.}, Theorem \ref{th:steady-states-normally-stable-hyperbolic}, and Corollary \ref{cor:uniform-bound-trajectories}, we have that $\Omega(u)$ is characterized by either normally stable or normally hyperbolic equilibria. We then apply Theorems \ref{th:stable} and \ref{th:unstable}, part \ref{ua2},\footnote{We note that, in the case $p=2$ and $\alpha=1$, the exponential decay from \eqref{eq:limit-H^2}.} to the solution 
\[
u(t):=(v(t+t_1),\omega(t+t_1),\gamma(t+t_1))
\]
of \eqref{eq:motion-abstract}, where $t_1$ is the time found in Corollary \ref{cor:uniform-bound-trajectories}. Finally, we take $t_\infty\ge t_1$. 
\end{proof}

We conclude this section by remarking that this last result improves considerably the results in \cite[Theorem 6.4]{GaldiMazzoneMohebbi18} and \cite[Theorems 7.1 \& 7.3]{GaldiMazzone21}. Our theorem shows convergence of generic trajectories (associated to weak solutions of \eqref{eq:motion}) to a steady state in more general topologies and with an exponential decay rate. Conditions \eqref{eq:condition-i.d.} on the data are not needed in \cite[Theorem 6.4]{GaldiMazzoneMohebbi18}, but it is exactly what is used to further characterize the stability properties of permanent rotations in \cite[Theorems 7.1 \& 7.3]{GaldiMazzone21}. It is an open questions whether conditions \eqref{eq:condition-i.d.} are also necessary for the exponential decay rate. Finally, our theorem holds for any possible choice of $\lambda_1$, $\lambda_2$ and $\lambda_3$ other than $\lambda_1=\lambda_2\ne \lambda_3$. It remains an open question whether exponential convergence to an equilibrium could also be proved in the case $\lambda_1=\lambda_2\ne\lambda_3$; we leave this to a future work. 

\bmhead{Acknowledgements}
The first and third authors gratefully acknowledge the support of the Natural Sciences and Engineering Research Council of Canada (NSERC) through the NSERC Discovery Grants  RGPIN-2022-04330 and RGPIN-2021-03129.

Giusy Mazzone is a member of ``Gruppo Nazionale per l'Analisi Matematica, la Probabilit\`a e le loro Applicazioni'' (GNAMPA) of Istituto Nazionale di Alta Matematica Francesco Severi (INdAM).

\begin{appendices}
\section{Re-parametrization of the equilibria set}\label{sec:flattening-equilibria}
In this section we reproduce step $(b)$ of the proof of \cite[Theorem 6.1]{PSZ09}. To this aim, let us recall the definitions of the spectral projections $P^\ell$, for $\ell\in\{c,s,u\}$, and of the spaces $X^\ell=P^\ell X$, from Remarks \ref{rm:spectral-projections-n-stable} and \ref{rm:spectral-projections-n-hyperbolic}. 

\begin{lemma}\label{lem:re-parametrization}
    There exist $\rho_0 >0$ and \( \Phi \in C^{1}(B_{X^{c}}(0, \rho_{0}); X_{1}) \), \( \Phi(0) = 0 \) such that the mapping \[x\in B_{X^{c}}(0, \rho_{0}) \mapsto \bar{u} + \Phi(x) \in \mathcal E\cap W\] parametrizes the manifold \( \mathcal{E} \) of equilibria near \( \bar u \). 

    Furthermore, for all $\ell\in\{s,u\}$, we have \( \phi_{\ell}:= P^{\ell} \Phi \in C^{1}(B_{X^{c}}(0, \rho_{0});X_{1}^{\ell}) \), and $ \phi_{\ell}(0) = \phi'_{\ell}(0) =0$. 
\end{lemma}
\begin{proof}
 From Definition \ref{def:normally-stable} (respectively, Definition \ref{def:normally-hyperbolic}), parts \ref{equi1}, \ref{tangent-space}, and \ref{semi-simple}, it follows that $X^c=N(L)= T_{\bar u}\mathcal E$.
 
Consider the mapping $g : U \subset \mathbb R^{m} \to X^{c}$ defined by
    \begin{align}\label{eq:def-g}
       g(\zeta) \:= P^{c}(\psi(\zeta)-\bar u).
    \end{align} 
We immediately note that $g(0)=0$. Since the rank of \( \psi'(0) \) is \( m \), it follows that \( g'(0) = P^{c} \psi'(0) \colon \mathbb{R}^{m} \to X^{c} \) is an isomorphism between the finite-dimensional spaces \( \mathbb{R}^{m} \) and \( X^{c} \). By the inverse function theorem, \( g \) is a \( C^{1} \)-diffeomorphism of a neighborhood of $0$ in $\R^m$, say \( \tilde U \subset U \), into a neighborhood of \( 0 \) in \( X^{c} \), say \( B_{X^{c}}(0, \rho_{0}) \). Let \( g^{-1} \colon B_{X^{c}}(0, \rho_{0}) \to \tilde U \) be the inverse mapping. Then \( g^{-1} \colon B_{X^{c}}(0, \rho_{0}) \to \tilde U \) is also \( C^{1} \)-mapping and \( g^{-1}(0) = 0 \). Next, we set 
\begin{equation}\label{eq:Phi}
\Phi(x) := \psi(g^{-1}(x)) -\bar u\qquad\text{for \( x \in B_{X^{c}}(0, \rho_{0}) \)}.
\end{equation}
Then, \( \Phi \in C^{1}(B_{X^{c}}(0, \rho_{0}); X_{1}) \), \( \Phi(0) = 0 \), $R(\Phi'(0)) \subset R(\psi'(0)) \subset X^c$, and
\begin{center}
    \( \{\Phi(x) + \bar u: x \in B_{X^{c}}(0, \rho_{0})\} = \mathcal{E} \cap W \),
\end{center}
 where \( W \) is an appropriate neighborhood of 0 in \( X_{1} \). 
 
By equations \eqref{eq:def-g} and \eqref{eq:Phi}, we note that, for all \( x \in B_{X^{c}}(0, \rho_{0}) \),
\[
P^{c} \Phi(x) = \left( (P^{c} \circ \psi) \circ g^{-1} \right)(x)-P^c\bar u = (g \circ g^{-1})(x) = x.
\]
Therefore, for all \( x \in B_{X^{c}}(0, \rho_{0}) \), we have
\begin{equation*}
    \begin{aligned}
        \Phi(x) &= P^{u} \Phi(x) + P^{c} \Phi(x) + P^{s} \Phi(x) \\
        &= P^{u} \Phi(x) + x + P^{s} \Phi(x).
    \end{aligned}
\end{equation*}
In addition, \( \phi_{\ell} = P^{\ell} \Phi \in C^{1}(B_{X^{c}}(0, \rho_{0}); X_{1}^{\ell}) \), \( \phi_{\ell}(0) = 0 \) and $\phi'_{\ell}(0) =P^\ell\Phi'(0)=0$, for all $\ell\in\{s,u\}$. 
\end{proof}

\section{Instability of normally hyperbolic equilibria}\label{app:unstable}
In this section, we will prove the first part of Theorem \ref{th:unstable} concerning the (nonlinear) instability of normally hyperbolic equilibria of \eqref{eq:evolution0}.
\begin{theorem}\label{th:unstable-correct-henry}
Let $(A,F)$ satisfy the hypotheses \ref{H1} and \ref{H2}. If $\bar u$ is a normally hyperbolic equilibrium of \eqref{eq:evolution}, then it is unstable.
\end{theorem}
\begin{proof}
Consider equation \eqref{eq:evolution1}, with $(L,G)$ defined in Remark \ref{rm:stability-0}, we show that the trivial solution of \eqref{eq:evolution1} is unstable. Let \(\sigma_1 := \sigma_u(L)\) (see \eqref{eq:unstable-spectrum}) and \(\sigma_2 := \sigma(L) \setminus \sigma_1\). By the spectral decomposition theorem (see Remark \ref{rm:spectral-projections-n-hyperbolic}), the space \(X\) admits a decomposition \(X = Y_1 \oplus Y_2\), where each \(Y_j\) is invariant under \(L\), and the spectrum of the restriction \(L_j := L|_{Y_j}\) satisfies \(\sigma(L_j) = \sigma_j\), for \(j = 1, 2\). We denote by \(E_j\) the spectral projection associated with the spectral set \(\sigma_j\), for $j=1,2$.

By the semigroup properties \eqref{eq:estimate-e-Ls-X}, \eqref{eq:estimate-e-Ls-Xbeta}, and the analiticity of the semigroup $\{\exp^{tL_1}\}_{t\ge 0}$, there exist constants \(\beta > 0\) and \(M_\alpha \geq 1\) such that the following estimates hold:
\begin{equation}\label{expo_bound_for_L_2}
\begin{aligned}
\|\mathrm{e}^{-tL_2} x\|_{X_\alpha} &\leq M_\alpha \|x\|_X\, t^{-\alpha} \mathrm{e}^{\beta t} \qquad &&\text{for }t> 0, \\
\|\mathrm{e}^{-tL_2} x\|_{X_\alpha} &\leq M_\alpha \|x\|_{X_\alpha}\, \mathrm{e}^{\beta t}\qquad &&\text{for }t> 0,
\end{aligned}
\end{equation}
and 
\begin{equation}\label{expo_bound_for_L_1}
\begin{aligned}
\|\mathrm{e}^{-tL_1} x\|_{X_\alpha} &\leq M_\alpha \|x\|_{X_\alpha}\, \mathrm{e}^{3\beta t}\qquad &&\text{for }t\le 0, \\
\|\mathrm{e}^{-tL_1} x\|_{X_\alpha} &\leq M_\alpha \|x\|_X\, \mathrm{e}^{3\beta t}\qquad &&\text{for }t\le 0.
\end{aligned}
\end{equation}

Let $a \in X_{1} \cap X_{\alpha}$ and $\tau>0$. Consider the map 
     \begin{multline}\label{eq:fixed-point-T}
      \mathcal T:\;y\mapsto 
      \mathcal T(y)(t):=  \exp^{-L_{1}(t-\tau)} a + 
      \int_{\tau}^{t} \exp^{-L_{1}(t - s)}E_{1}G(y(s)) ds
      \\
      + \int_{-\infty}^{t} \exp^{-L_{2}(t - s)}E_{2}G(y(s)) ds
      \quad\text{for }t\le \tau. 
     \end{multline}
With a suitable choice of $a$ and $\rho>0$, we prove that $\mathcal T$ admits a fixed point in \[
\mathcal{A} := \left\{ f \in C((-\infty, \tau]; X_\alpha) \, : \, E_1 f(\tau) = a \text{ and } \|f(t)\|_{X_\alpha} \leq \rho \exp^{2\beta (t - \tau)} \text{ for all } t \leq \tau \right\},
\]
equipped with the norm
\[
\|f\|^S := \sup_{t \leq \tau} \left\{ \exp^{-2\beta (t - \tau)} \|f(t)\|_{X_\alpha} \right\}.
\]
It is straightforward to verify that \(\mathcal{A}\), with the norm \(\|\cdot\|^S\), is a complete metric space. In addition, we note that a fixed point of $\mathcal T$ is indeed a solution of the ODE \eqref{eq:evolution1}$_1$. Next, we show that \(\mathcal T\) maps \(\mathcal{A}\) into itself and is a strict contraction with respect to the norm \(\|\cdot\|^S\).

Let \( y \in \mathcal{A} \). We estimate the two integral terms in \eqref{eq:fixed-point-T} using \eqref{expo_bound_for_L_1}, \eqref{expo_bound_for_L_2} and standard inequalities. First, we find that 
\begin{align}
\left\| \int_\tau^t \exp^{-L_1(t - s)} E_1 G(y(s))\, ds \right\|_{X_\alpha}
&\leq M_\alpha \displaystyle\int_\tau^t \exp^{3\beta(t - s)} \|E_1\|_{\mathcal B(X;X)}\, k(\rho) \|y(s)\|_{X_\alpha} \, ds \nonumber\\
&\leq M_\alpha \|E_1\|_{\mathcal B(X;X)}\, k(\rho)\, \rho \displaystyle\int_\tau^t \exp^{3\beta(t - s)} \exp^{2\beta(s - \tau)}\, ds \nonumber \\
&\leq M_\alpha\, k(\rho)\,  \frac{\rho}{\beta}\|E_1\|_{\mathcal B(X;X)} \, \exp^{2\beta(t - \tau)},\label{2.29}
\end{align}
where 
\[
k(\rho):=\sup_{x\in B_{X_\alpha}\!(0,\rho)}\norm{G'(x)}_{\mathcal B(X_\alpha;X)}.
\]
In a similar fashion, we estimate the second integral in \eqref{eq:fixed-point-T}:
\begin{align} 
&\left\| \int_{-\infty}^t \exp^{-L_2(t - s)} E_2 G(y(s))\, ds \right\|_{X_\alpha}
\leq M_\alpha \displaystyle\int_{-\infty}^t \frac{\exp^{\beta(t - s)}}{(t - s)^\alpha} \|E_2\|_{\mathcal B(X;X)}\, k(\rho) \|y(s)\|_{X_\alpha}\, ds \nonumber\\
&\leq M_\alpha\, \|E_2\|_{{\mathcal B(X;X)}}\, k(\rho)\, \rho\, \exp^{2\beta(t - \tau)} \displaystyle\int_0^\infty \frac{\exp^{-\beta s}}{s^\alpha}\, ds.\label{2.30}
\end{align}
Combining estimates \eqref{2.29} and \eqref{2.30}, we obtain that
\begin{equation}
\begin{aligned}
\|\mathcal T(y)(t)\|_{X_\alpha}
&\leq M_\alpha \exp^{3\beta(t - \tau)} \|a\|_{X_\alpha} + M_\alpha\, k(\rho)\, \frac{\rho}{\beta}\, \|E_1\|_{{\mathcal B(X;X)}} \exp^{2\beta(t - \tau)} \\
&\quad + M_\alpha\, \|E_2\|_{\mathcal B(X;X)}\, k(\rho)\, \rho\, \exp^{2\beta(t - \tau)} \int_0^\infty \frac{\exp^{-\beta s}}{s^\alpha}\, ds \\
&\leq
\rho \exp^{2\beta(t - \tau)}
    \qquad\qquad\text{for all \( t \in (-\infty, \tau] \)},
\end{aligned}
\end{equation}
provided that \(\rho > 0\) is sufficiently small so that
\[
M_\alpha k(\rho) \left( \frac{1}{\beta}\norm{E_1}_{\mathcal B(X;X)} + \|E_2\|_{\mathcal B(X;X)} \int_0^\infty \frac{\exp^{-\beta s}}{s^\alpha} \, ds \right) \leq \frac{1}{4M_\alpha} < \frac{1}{2}
\]
and \(a \in X_1 \cap X_\alpha\) satisfies \(\|a\|_{X_\alpha} = \frac{\rho}{2M_\alpha}\). 
Thus, \(\mathcal T\) maps \(\mathcal{A}\) into itself, and it is actually a contraction on $\mathcal A$ if additionally $\rho<1$. 

By Banach fixed point theorem, the map $\mathcal T$ in \eqref{eq:fixed-point-T} has a unique fixed point, say  \(t\mapsto y^*(t)=y^*(t; \tau, a)\in \mathcal A \). From the definition of $\mathcal A$ and the choice of $a$, it follows that
\[
\|y^*(t)\|_{X_\alpha} \leq 2 M_\alpha \|a\|_{X_\alpha} \exp^{2\beta(t - \tau)}.
\]
Moreover, by \eqref{2.30},
\[
\left\| \int_{-\infty}^\tau \exp^{-L_2(\tau - s)} E_2 G(y^*(s))\, ds \right\|_{X_\alpha} \leq \frac{1}{2} \|a\|_{X_\alpha}.
\]
From the latter estimate, and \eqref{eq:fixed-point-T} at $t=\tau$ and $y=y^*$, we obtain
\[
\| y^*(\tau) \|_{X_\alpha} \geq \frac{1}{2} \| a \|_{X_\alpha}.
\]
Recall that \( y^*(\cdot; \tau, a) \) is the solution of the ODE \eqref{eq:evolution1}$_1$ for $t<\tau$. 
Now, for each \( n \in \mathbb{N} \), define \( z_n := y^*(0; n, a) \). Let \(t\in [0,t^*)\mapsto u(t):=z(t; 0, z_n) \) be the solution of \eqref{eq:evolution1} with initial condition \( z_n \), where $t^*$ is the maximal existence time. If $t^*<\infty$, then the trivial solution of \eqref{eq:evolution1} is unstable. Assume that $t^*=\infty$. Then $u$ satisfies
\[
z(t; 0, z_n) = y^*(t; n, a)
\qquad\text{ for \( t\in [0,n]\)},
\]
and the following estimates:
\[
\begin{aligned}
\sup_{t \geq 0} \| z(t; 0, z_n) \|_{X_\alpha} 
&\geq \| z(n; 0, z_n) \|_{X_\alpha} \\
&= \| y^*(n; n, a) \|_{X_\alpha} \\
&\geq \frac{1}{2} \|a\|_{X_\alpha}.
\end{aligned}
\]
On the other hand, for every $R\in \left(0,\displaystyle\frac{\norm{a}_{X_\alpha}}{4}\right)$, there exists $N\in \N$ sufficiently large so that 
\[
\| z_N \|_{X_\alpha} = \| y^*(0; N, a) \|_{X_\alpha} \leq \exp^{-2\beta N}< R.
\]
Therefore, the trivial solution of \eqref{eq:evolution1} is unstable.
\end{proof}
\end{appendices}

\bibliography{sn-bibliography}

\end{document}